\documentclass[12pt]{article}
\usepackage[centertags]{amsmath}
\usepackage{amsfonts}
\usepackage{amssymb}
\usepackage{amsthm}
\usepackage{amsmath,mathrsfs}
\usepackage{amsmath,amscd}
\title{\textbf{Higher Tannaka and Beyond}}
\author{Renaud Gauthier \footnote{rg.mathematics@gmail.com} \\ \\}
\theoremstyle{definition}
\newtheorem*{acknowledgments}{Acknowledgments}
\newtheorem{Wall}{Theorem}[section]

\newtheorem{InftyCat}{Definition}[section]
\newtheorem{modelCat}[InftyCat]{Definition}
\newtheorem{localization}[InftyCat]{Definition}
\newtheorem{Quillen}[InftyCat]{Definition}
\newtheorem{MonModCat}[InftyCat]{Definition}
\newtheorem{SimpModCat}[InftyCat]{Definition}
\newtheorem{WeaklySat}[InftyCat]{Definition}
\newtheorem{CombModCat}[InftyCat]{Definition}
\newtheorem{presentable}[InftyCat]{Definition}
\newtheorem{Rmkpres}[InftyCat]{Remark}
\newtheorem{SymmMonCat}[InftyCat]{Definition}
\newtheorem{Tspectrum}{Definition}[section]
\newtheorem{spectrum}[Tspectrum]{Definition}
\newtheorem{ComMonObj}[Tspectrum]{Definition}
\newtheorem{CRingSpec}[Tspectrum]{Definition}
\newtheorem{Rlinear}[Tspectrum]{Definition}
\newtheorem{stack}{Definition}[section]
\newtheorem{finite}[stack]{Definition}
\newtheorem{Rmrk1}[stack]{Remark}
\newtheorem{HopfRAlg}[stack]{Definition}
\newtheorem{finHopf}[stack]{Definition}
\newtheorem{RTensor}{Definition}[section]
\newtheorem{RTan}[RTensor]{Definition}
\newtheorem{pointed}[RTensor]{Definition}
\newtheorem{binding}{Definition}[subsection]
\newtheorem{Propbinding}[binding]{Proposition}
\newtheorem{Conj}{Conjecture}[subsection]
\newtheorem{Xcount}{Definition}[subsection]
\newtheorem{metric}[Xcount]{Definition}
\newtheorem{blowup}{Definition}[subsection]
\newtheorem{blowup2}[blowup]{Definition}
\newtheorem{crystal}[blowup]{Definition}
\newtheorem{tangentsheaf}[blowup]{Definition}
\newtheorem{Sinfty}[blowup]{Definition}
\newtheorem{SinftyConj}[blowup]{Conjecture}
\newtheorem{sudden}[blowup]{Proposition}
\newtheorem{FCatConn}[blowup]{Conjecture}
\newtheorem{GrothTop}{Proposition}[subsection]
\newtheorem{ZGrothn}[GrothTop]{Definition}
\newtheorem{ZGroth}[GrothTop]{Definition}
\newtheorem{PropGroth}[GrothTop]{Proposition}
\newcommand{\beq}{\begin{equation}}
\newcommand{\eeq}{\end{equation}}
\begin{document}
\maketitle
\begin{abstract}
We consider the origins of Higher Tannaka duality, as well as it consequences. In a first time we review the work of J. Wallbridge (\cite{W}) on that subject, which shows in particular that Hopf algebras are essential to generating Tannakian $\infty$-categories. While doing so we provide an analysis of what it means for Higher Tannaka to be a reconstruction program for stacks. Putting Wallbridge's work in perspective paves the way for causal models in Higher Category Theory. We introduce two interesting concepts that we deem to be instrumental within this framework; that of blow-ups of categories as well as that of fractal $\infty$-categories.
\end{abstract}
\newpage

\section{Introduction}
The origin of Tannaka duality can be found in \cite{C}, and from the perspective of a reconstruction program, the duality theorems of Tannaka and Krein were the first to show that some algebraic objects such as compact groups could be recovered from the collection of their representations. Tannaka (\cite{Ta}) showed that a compact group can be recovered from its category of representations. Krein (\cite{K}) worked on those categories that arise as the category of representations of such groups. A nice account of Tannaka duality can be found in \cite{JS}, with a formal development of the basis for a modern form of Tannaka duality given in \cite{DM}. It is worth giving the main result of that latter paper to give a flavor of the Tannaka formalism. Deligne and Milne considered a rigid abelian tensor category $(C, \otimes)$ for which $k=\text{End}(\textbf{1})$ and let $\omega: C \rightarrow \text{Vec}_k$ be an exact faithful $k$-linear tensor functor. Such a pair $(C, \omega)$ is otherwise known as a \textbf{neutralized Tannakian category}. They showed that Aut$^{\otimes}(\omega)$ is represented by an affine group scheme $G$ and that the functor $C \rightarrow \text{Rep}_k(G)$ given by the \textbf{fiber functor} $\omega$ is an equivalence of tensor categories. Papers of importance to us are given by \cite{M} which introduces Hopf algebras into the picture. One cannot omit \cite{SR} from the list of references, the first place where affine group schemes are introduced in the Tannaka formalism. A first step towards a higher categorical generalization of this formalism is made in \cite{T1} and is further developed in \cite{T3}. The use of ring spectra and affine group stacks is exemplified in \cite{W} which we use as a basis for studying the causality of Tannaka duality theorems. Extension of the Tannaka formalism to that of geometric stacks has been done by Lurie \cite{L3} as well as in the more thorough \cite{L4}. An interesting work in the continuation of \cite{W} is that of \cite{I}, where an application to motives is made.\\

We regard $(\infty,1)$-categories, heretofore referred to as $\infty$-categories, as building blocks of any natural phenomenon. In other terms we are ultimately interested in the phenomenology of higher Tannakian categories insofar as once one knows what they do, one can essentially determine what they are. We ask the simple question: given an $\infty$-Tannakian category and one of its representations, is there some operation generating such a representation. The answer to this question can be found in \cite{W}. There is a functor from the $\infty$-category of certain group stacks to a particular $\infty$-category of pointed $\infty$-categories. The group stacks $G=\text{Spec}(B)$ in question, $B$ being a certain kind of Hopf algebra, are thus seen as being essential to the Tannakian formalism. We reformulate this result as a reconstruction problem: if $\mathcal{X}$, $\mathcal{M}$ are stacks, $\pi: \mathcal{X} \rightarrow \mathcal{M}$ is a morphism of stacks, we seek to find stacks $\mathcal{U}^{(i)}$, and morphisms $\omega^{(i)}$, $i \geq 1$, such that in the following sequence:
\beq
\setlength{\unitlength}{0.5cm}
\begin{picture}(18,12)(0,0)
\put(0.2,10){$\mathcal{X}$}
\put(0.5,9.7){\vector(0,-1){1}}
\put(-0.3,9.2){$\pi$}
\put(-2,8){$\mathcal{\mathcal{M} \ni \pi \mathcal{X}}$}
\put(1.5,8.2){\vector(1,0){2.5}}
\put(4.5,8){$\mathcal{U}^{(1)}$}
\put(5,7.5){\vector(0,-1){2}}
\put(5.5,6.5){$\omega^{(1)} \in \mathcal{U}^{(3)}$}
\put(4.5,4.5){$\mathcal{U}^{(2)}$}
\put(8.5,6){\vector(0,-1){2}}
\put(8,3){$\mathcal{U}^{(4)}$}
\put(9,4.7){$\omega^{(2)}$}
\multiput(11,4)(1,-0.5){4}{\circle*{0.2}}
\put(15,2.8){$\mathcal{U}^{(p)}$}
\put(15.5,2.5){\vector(0,-1){1.8}}
\put(16,1.6){$\omega^{((p+1)/2)}$}
\put(15,-0.2){$\mathcal{X}$}
\end{picture}\\ \\ \nonumber
\eeq
\newline
we eventually recover $\mathcal{X}$ from the knowledge of $\pi \mathcal{X}$ only. In \cite{W} Higher Tannaka duality is exemplified with the very short sequence:
\beq
\setlength{\unitlength}{0.5cm}
\begin{picture}(12,7)(0,0)
\put(0.2,6){$G$}
\put(0.5,5.5){\vector(0,-1){1.5}}
\put(0.7,4.8){$\tilde{B}$}
\put(0,3){$\tilde{B}G$}
\put(1.5,3.2){\vector(1,0){1.5}}
\put(3.2,3){$\text{Mor}(\tilde{B}G, \text{Perf})$}
\put(5,2.5){\vector(0,-1){1.5}}
\put(5.2,1.8){$\omega$}
\put(3.2,0){$\text{Mor}( \ast, \text{Perf})$}
\put(6,2){\vector(1,0){2}}
\put(8.2,1.8){$\text{End}^{\otimes} \omega \in \text{TGp}^{\tau}(R)$}
\put(9,0.8){$\parallel$}
\put(8.9,-0.2){$G$}
\end{picture}\\ \\ \nonumber
\eeq
\newline
We then shift our attention to studying non-circular sequences of such morphisms and view them as representative of a causality in Algebraic Geometry. This motivates the introduction of categories with adjunctions between them and morphisms above them such as in:
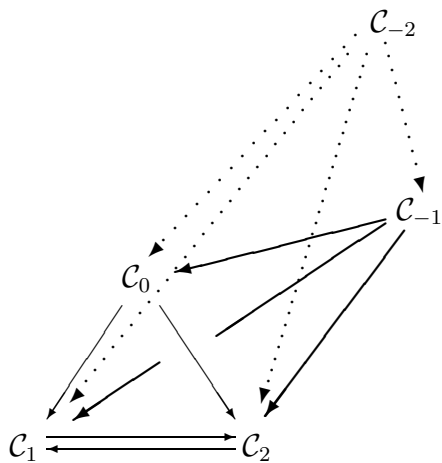
\beq
\setlength{\unitlength}{0.5cm}
\begin{picture}(12,14)(0,0)
\put(-1,1){$\mathcal{C}_1$}
\put(0,1.5){\vector(1,0){5}}
\put(5,1.2){\vector(-1,0){5}}
\put(5.2,1){$\mathcal{C}_2$}
\put(2,5.5){$\mathcal{C}_0$}
\put(2,5){\vector(-2,-3){2}}
\put(3,5){\vector(2,-3){2}}
\put(9.3,7.3){$\mathcal{C}_{-1}$}
\thicklines
\put(9.5,7){\vector(-3,-4){3.7}}
\put(9,7.2){\line(-3,-2){4.5}}
\put(3,3.5){\vector(-3,-2){2.3}}
\put(9,7.3){\vector(-4,-1){5.6}}
\put(8.6,12.3){$\mathcal{C}_{-2}$}
\multiput(9,12)(0.09,-0.36){10}{\circle*{0.1}}
\put(9.9,8.4){\vector(1,-4){0.1}}
\multiput(8.5,11.7)(-0.09,-0.3){30}{\circle*{0.1}}
\put(5.8,2.8){\vector(-1,-4){0.1}}
\multiput(8,11.7)(-0.24,-0.3){30}{\circle*{0.1}}
\put(0.7,2.5){\vector(-1,-1){0.1}}
\multiput(8.2,12.2)(-0.3,-0.32){18}{\circle*{0.1}}
\put(2.8,6.4){\vector(-1,-1){0.1}}
\end{picture} \nonumber
\eeq
the collection of which forms a category that we denote by Cat$^{\triangleq}$ and which we conjecture to be connected. Towards proving such a result we introduce the concept of blow-up of a category, very much in the same spirit as the algebro-geometric definition:
\beq
\begin{CD}
\widetilde{\mathcal{C}_c}=\mathcal{C}-c \bigcup_{\pi}\text{E}(c) \\
@VV \pi V \\
\mathcal{C}
\end{CD}
\nonumber
\eeq
with E$(c)=\{ c' \xrightarrow{\psi} c \quad | \quad c' \in \text{Ob}(\mathcal{C}), \text{ $\psi$ not a composition}\}$. \\

Another model we introduce is that of fractal $\infty$-categories with non-well founded set theory as a foundation. Those are $\mathbb{Z}$-graded category-like structures $\mathcal{C}^{(*)}$ such that for all $n$, $\mathcal{C}^{(n)}$ has $\mathcal{C}^{(n-1)}$ as an object as well as all $n$-1-st categories $\mathcal{C}^{(n-1)}$ has a morphism with. "Higher" morphisms are functors between "lower" morphisms which gives credence to the term $\infty$-category. The fractal-like presentation of such $\mathbb{Z}$-graded categories leads to the notion of fractal $\infty$-categories, the collection of which we also conjecture to be connected. Disclaimer: there is an obvious bias in the references as we have mentioned those that we deemed to be of greater importance. It does not mean that those we didn't mention are any less important.

\begin{acknowledgments}
The author would like to thanks D. Yetter for sharing its knowledge of Category Theory and related topics, and for enriching discussions. Thanks also go to J. Wallbridge for explicating parts of his paper as well as to J. Bergner for giving very helpful comments on a preliminary draft of this paper.
\end{acknowledgments}

\section{Wallbridge's Higher Tannaka Results with Generalizations}
If we consider higher Tannakian categories as being essential components of any viable model of nature, is there some mathematical object that produces such categories? The answer to that question is in the affirmative and can be recovered from \cite{W}. We provide the raw statement below. All notations and background material can be found in the appendix.
\begin{Wall}(\cite{W})
Let $\tau$ be a subcanonical topology, TGp$^{\tau}(R)$ the $(\infty,1)$-category of $R$-Tannakian group stacks, (Tens$^{\text{rig}}_R)_*$ the $(\infty,1)$-category of pointed rigid $R$-tensor $(\infty,1)$-categories. One has an adjunction:
\beq
\text{Fib}_*: (\text{Tens}^{\text{rig}}_R)_* \rightleftarrows \text{Gp}^{\tau}(R)^{op}:\text{Perf}_*
\eeq
Perf$_*$ is fully faithful. If $R$ is an $E_{\infty}$-ring, then $(T,\omega)$ is a pointed finite $R$-Tannakian $\infty$-category if and only if $(T,\omega)=$Perf$_*(G)$ for $G$ a finite $R$-Tannakian group stack.
\end{Wall}

\subsection{Fleshing out the details: Tannaka and beyond}
The ingredients of Wallbridge's higher Tannaka results are best understood from the perspective of Brave New Algebraic Geometry (BNAG for short), a specialization to stable homotopy theory of Homotopical Algebraic Geometry (HAG for short) (\cite{T2}, \cite{TV2}, \cite{TV3}, \cite{TV5}), a mathematical setting where affine objects are modelled by homotopy-ring like objects. In particular, in relative algebraic geometry one can do algebraic geometry over well-behaved symmetric monoidal base categories, and this program has been initiated in \cite{TV1} for symmetric monoidal $\infty$-categories. Thus working with a site such as $(\text{Aff}_R, \tau)$ comes naturally to mind if one is interested in working in BNAG. It is in \cite{T1} and \cite{T2} that we see some work done on $\text{Aff}_M:=(\text{Comm}(M))^{op}$, $M$ being a symmetric monoidal model category, the base category of some HAG-like theory. If one takes $M$ to be the category of spectra, one does BNAG. Thus working with E$_{\infty}$-rings is well-motivated, and in particular using $(\text{Aff}_R, \tau)$, $R$ an E$_{\infty}$-ring,  as a site is well-motivated as well. In \cite{T1} Toen considered Tannaka duality within the context of BNAG. Algebraically, one needs to consider algebraic group schemes. Hence the use of affine group stacks in Wallbridge's framework.\\

We look at Tannaka duality as a recovery problem (read reconstruction) that may be stated in broad terms as follows. We consider the very general problem of having a stack $\mathcal{X}$, a moduli stack $\mathcal{M}$, a morphism $\mathcal{X} \xrightarrow{\pi} \mathcal{M}$, and of having to recover $\mathcal{X}$ from the knowledge of $\pi \mathcal{X}$. In order to achieve this we adopt the point of view that it is necessary not to isolate $\pi \mathcal{X}$ but to compare it to other appropriately chosen stacks. If $\mathcal{U}$ is one such stack, a comparison is conveniently provided by the morphism space Mor($\pi \mathcal{X}, \mathcal{U})$. Most likely this will not be sufficient to determine $\mathcal{X}$ in full. Further comparisons may be required. Under certain conditions, Mor($\pi \mathcal{X}, \mathcal{U})$ itself is a stack that we can compare to another stack of our chosing. We let $\mathcal{U}:=\mathcal{U}^{(1)}$, Mor($\pi \mathcal{X}, \mathcal{U}):=\mathcal{U}^{(2)}$. If we have chosen another stack for comparison purposes, we denote it by $\mathcal{U}^{(3)}$. The comparison is then provided by considering Mor($ \mathcal{U}^{(2)}, \mathcal{U}^{(3)})$. We do this as many times as necessary in the hope that for some $p$ finite we end up having $\mathcal{X}$ as the image of some morphism $\omega \in \text{Mor}(\mathcal{U}^{(p-1)}, \mathcal{U}^{(p)})$. This motivates the following definition:
\begin{binding}
Let $\mathcal{X}$ be a stack, $\mathcal{M}$ a moduli stack, $\pi: \mathcal{X} \rightarrow \mathcal{M}$ a morphism of stacks. A \textbf{volume of stacks for} $\mathcal{X}$ is a sequence of stacks $\mathcal{U}^{(i)}$, $i \geq 1$ defined inductively by setting Mor($\mathcal{U}^{(n-1)}, \mathcal{U}^{(n)})=:\mathcal{U}^{(n+1)}$ with $\mathcal{U}^{(1)}$ being given and Mor($\pi \mathcal{X}, \mathcal{U}^{(1)})=:\mathcal{U}^{(2)}$. A given volume is called a \textbf{bound volume for} $\mathcal{X}$ if we can find a sequence of morphisms $\omega^{(n)} \in \text{Mor}(\mathcal{U}^{(2n)}, \mathcal{U}^{(2n+1)})$ such that for some $p$ finite the image of $\omega^{(p)} \circ \cdots \omega^{(1)}$ in $\mathcal{U}^{(2n+1)}$ is precisely $\mathcal{X}$. The \textbf{binding} is then given by the sequence of stacks $\mathcal{U}^{(i)}$ and morphisms $\omega^{(i)}$. Failure to recover $\mathcal{X}$ from a volume gives us a \textbf{loose volume}. We graphically represent a bound volume as follows:
\beq
\setlength{\unitlength}{0.5cm}
\begin{picture}(18,12)(4,0)
\put(0.2,10){$\mathcal{X}$}
\put(0.5,9.7){\vector(0,-1){1}}
\put(-0.3,9.2){$\pi$}
\put(0,8){$\mathcal{M}$}
\put(1.2,8.2){\vector(1,0){1}}
\put(2.5,8){$\text{Mor}(\pi \mathcal{X}, \mathcal{U}^{(1)}):=\mathcal{U}^{(2)}$}
\put(5,7.5){\vector(0,-1){2}}
\put(5.5,6.5){$\omega^{(1)} \in \text{Mor}(\mathcal{U}^{(2)}, \mathcal{U}^{(3)}):=\mathcal{U}^{(4)}$}
\put(4.5,4.5){$\mathcal{U}^{(3)}$}
\put(9,6){\vector(0,-1){2}}
\put(8.5,3){$\mathcal{U}^{(5)}$}
\put(9.5,4.7){$\omega^{(2)} \in \text{Mor}(\mathcal{U}^{(4)}, \mathcal{U}^{(5)}):=\mathcal{U}^{(6)}$}
\multiput(12,3)(1,-0.5){4}{\circle*{0.2}}
\put(15,2.8){$\text{Mor}(\mathcal{U}^{(2n-2)}, \mathcal{U}^{(2n-1)}):=\mathcal{U}^{(2n)}$}
\put(17,2.5){\vector(0,-1){2}}
\put(17.2,1.6){$\omega^{(n)}$}
\put(15,-0.2){$\mathcal{X} \in \mathcal{U}^{(2n+1)}$}
\end{picture}\\ \\
\eeq \nonumber \\
\end{binding}
What \cite{W} gives is a bound volume for Tannakian group stacks as we will see. We see the Tannaka reconstruction as essentially recovering finite-$R$-Tannakian group stacks $G$ from their classifying stack $\tilde{B}G$. It is true that $\tilde{B}$ has an adjoint $\tilde{\Omega}:\text{St}^{\tau}(R) \rightarrow \text{Gp}^{\tau}(R)$, that sends a stack $F$ to $\tilde{\Omega}(F):x \mapsto \omega(F(x))$ where for an $(\infty,0)$-category $X$, $\omega(X): [n] \mapsto X^{\Delta^n_{\ast}}$. What Wallbridge does is much more; he finds an inverse to $\tilde{B}$ by comparing $\tilde{B}G$ to Perf, and then comparing Mor($\tilde{B}G, \text{Perf})$ to Mor($\ast, \text{Perf})$. In doing so Wallbridge does not see the morphisms spaces as stacks but as $R$-tensor categories instead. The last comparison is induced from the natural morphism $\ast \rightarrow \tilde{B}G$ and is denoted $\omega$. He proves End$^{\otimes}\omega = G$.
\begin{Propbinding}
For $\mathcal{X}$ a finite-$R$-Tannakian group stack, $\pi=\tilde{B}$, $\mathcal{M}=\tilde{B}G$, one has a bound volume for $G$ being provided by choosing $\mathcal{U}^{(1)}=\text{Perf}$, $\mathcal{U}^{(3)}=\text{Mor}(\ast, \text{Perf})$, $\omega^{(1)}=\omega$ as defined above and $\omega^{(2)}$ is the morphism that sends any morphism $\psi$ to End$^{\otimes} \psi$.
\end{Propbinding}
\begin{proof}
This is a simple rewriting of Theorem 7.14 of \cite{W}. Graphically:
\beq
\setlength{\unitlength}{0.5cm}
\begin{picture}(12,7)(0,0)
\put(0.2,6){$G$}
\put(0.5,5.5){\vector(0,-1){1.5}}
\put(0.7,4.8){$\tilde{B}$}
\put(0,3){$\tilde{B}G$}
\put(1.5,3.2){\vector(1,0){1.5}}
\put(3.2,3){$\text{Mor}(\tilde{B}G, \text{Perf})$}
\put(5,2.5){\vector(0,-1){1.5}}
\put(5.2,1.8){$\omega$}
\put(3.2,0){$\text{Mor}( \ast, \text{Perf})$}
\put(6,2){\vector(1,0){2}}
\put(8.2,1.8){$\text{End}^{\otimes} \omega \in \text{TGp}^{\tau}(R)$}
\put(9,0.8){$\parallel$}
\put(8.9,-0.2){$G$}
\end{picture}  \nonumber\\ \\
\eeq
\end{proof}

\subsection{Analysis}
Higher Tannaka as exposed by Wallbridge is a surprisingly fast way to recover a finite-$R$-Tannakian group stack $G$ from $\tilde{B}G$. This is owing to the fact that End$^{\otimes}\omega$ is a representable Gp$(\mathcal{S})$-valued prestack, and thus an affine group stack with respect to \textit{any} subcanonical topology (Proposition 6.10 of \cite{W}). That justifies that one can close our bound volume in two steps only. The proof of such a statement hinges on two nontrivial facts. The first is Lemma 6.9 whose proof involves arguing that symmetric monoidal $\infty$-categories are given by the homotopy colimit of free symmetric monoidal $\infty$-categories over $\infty$-graphs (\cite{W}), conveniently defined by a universal property. This introduces maps into arbitrary symmetric monoidal $\infty$-categories in the picture. If one picks Mod$^{\text{rig}}_R$ to be one such category, one ends up having an equivalence End$^{\otimes}\omega \simeq \text{End}(M)$ for $M$ a rigid module, and being dualizable, this latter object can be written Spec(Fr($M \otimes_R M^{\vee}))$ where Fr: Mod$_R \rightarrow \text{CAlg}_R$ is the free functor, left adjoint to the forgetful functor CAlg$_R \rightarrow  \text{Mod}_R$. The other fact that is needed and is not obvious is that End$^{\otimes}\omega \simeq \mathbb{R}\text{Aut}^{\otimes}\omega$, and this is true by virtue of Proposition 6.4 of \cite{W}(originally from \cite{Sa}).\\

Proposition 4.6 of \cite{W} is a pivotal point where we move from discussing stacks $\tilde{B}G$ and Perf to discussing symmetric monoidal $\infty$-categories, a setting for which the $\infty$-graph discussion takes place. This proposition states that for the site $(\text{Aff}_R, \tau)$ and for Cat$_{\infty}$, a presentable symmetric monoidal $\infty$-category, then St$_{\text{Cat}_{\infty}}(R)$ is tensored and enriched over Cat$_{\infty}$. That this latter category is presentable is owing to the fact that it is the localization of a combinatorial simplicial model category, $P(\Sigma)$, a fact reminded to us by Wallbridge. From that point on one can use the adjunction $\mathcal{S} \rightleftarrows \text{Cat}_{\infty}$ to see $\tilde{B}G$ as an element not of St$^{\tau}(R)$ but of St$^{\tau}_{\text{Cat}_{\infty}}(R)$, from which it follows that one can invoke the fact that Mor$(\tilde{B}G, \text{Perf})$ really is an element of Cat$_{\infty}$. From there it soon follows that it is also an element of Tens$^{\text{rig}}_R$ and Wallbridge's program comes into effect.\\

If we come back to our general question of recovering a stack $\mathcal{X}$ from the knowledge of $\pi \mathcal{X}$, suppose we restrict ourselves to working in Homotopical Algebraic Geometry. Suppose further we are interested in Stable Homotopy Theory. Then it is natural to work within the confines of Brave New Algebraic Geometry. Thus we can consider $R$ to be an E$_{\infty}$-ring and we can work on a site $(\text{Aff}_R, \tau)$ as well, over which we construct stacks. $\mathcal{X}$ and $\mathcal{M}$ will be stacks over such a site. Suppose we are given $\pi \mathcal{X} \in \mathcal{M}$. We initiate a volume for $\mathcal{X}$. A simple first stack to compare $\pi \mathcal{X}$ to is Perf. We consider Mor($\pi \mathcal{X}, \text{Perf})$. If further $\mathcal{M} \in \text{St}^{\tau}(R)$, then by Proposition 4.6 of \cite{W}, Mor($\pi \mathcal{X}, \text{Perf})$ can be viewed as an element of Cat$_{\infty}$, and following the argument of Wallbridge we find that it is an element of Tens$^{\text{rig}}_R$. $\omega: \text{Mor}(\pi \mathcal{X}, \text{Perf}) \rightarrow \text{Mor}(\ast, \text{Perf})$ we regard as a comparison that comes naturally to mind even if we do not see it as being induced by the morphism $\ast \rightarrow \pi \mathcal{X}$. From our persective it amounts to comparing morphisms $\pi \mathcal{X} \rightarrow \text{Perf}$ with the very elementary morphisms $\ast \rightarrow \text{Perf}$. We impose that $\omega$ be a finite-fiber functor. Then the pair $(\text{Mor}(\pi \mathcal{X}, \text{Perf}), \omega)$ is a pointed finite $R$-Tannakian $\infty$-category and Theorem 7.14 of \cite{W} tells us we must have Mor($\pi \mathcal{X}, \text{Perf})=\text{Mor}(\tilde{B}G, \text{Perf})$ for $G$ a finite $R$-Tannakian group stack. Thus what we obtain at best is a local equivalence of $\pi \mathcal{X}$ with a classifying stack $\tilde{B}G$ of a Tannakian group stack $G$. Further End$^{\otimes}\omega=G$, at the second stage in the volume for $\mathcal{X}$. It means an additional two comparisons such as above yields back $G$ again and thus provides a bound volume for $G$. \\

In the next section we go beyond Tannaka as this duality yields bound volumes for Tannakian group stacks. What we are really interested in is loose volumes. If we want to determine the origin of any given viable model of nature one does not want to deal with bound volumes insofar as those do not provide an origin but simply display a circularity in ideas that though interesting from a duality perspective is not very illuminating for the purpose of answering fundamental questions about causality.

\section{Causality and Dynamics in Higher Categories}
\subsection{Conic towers of categories}
Let Cat denote the category of all small categories. From Cat we construct another category Cat$^{\triangleq}$ defined as follows. For objects we take small categories that have an adjunction with at least one other small category, and for which we can find a third category with a morphism to each of those two categories. If $\mathcal{C}_1$ and $\mathcal{C}_2$ are two such categories between which there is an adjunction, $\mathcal{C}$ is a category with a morphism to both categories, then both $\mathcal{C}_1$ and $\mathcal{C}_2$ are viewed as objects of Cat$^{\triangleq}$ (though $\mathcal{C}$ may not be so as it may not necessarily have an adjunction with another category). We first introduce morphisms in Cat$^{\triangleq}$ informally before defining them in full. The triangle depicted below is defined to be a morphism in Cat$^{\triangleq}$ between $\mathcal{C}_1$ and $\mathcal{C}_2$.

\beq
\setlength{\unitlength}{0.5cm}
\begin{picture}(7,6)(0,0)
\put(-1,1){$\mathcal{C}_1$}
\put(0,1.5){\vector(1,0){5}}
\put(5,1.2){\vector(-1,0){5}}
\put(5.2,1){$\mathcal{C}_2$}
\put(2,5.5){$\mathcal{C}$}
\put(2,5){\vector(-2,-3){2}}
\put(3,5){\vector(2,-3){2}}
\end{picture} \label{simplecone}
\eeq
This is the simplest morphism we can have. From such a morphism we can have another one if, writing $\mathcal{C}_0$ for $\mathcal{C}$ we can find a fourth category $\mathcal{C}_{-1}$ with a morphism to $\mathcal{C}_0$, and by composition it would also have a morphism to $\mathcal{C}_1$ and $\mathcal{C}_2$ as well. We depict this as follows:

\beq
\setlength{\unitlength}{0.5cm}
\begin{picture}(12,8)(0,0)
\put(-1,1){$\mathcal{C}_1$}
\put(0,1.5){\vector(1,0){5}}
\put(5,1.2){\vector(-1,0){5}}
\put(5.2,1){$\mathcal{C}_2$}
\put(2,5.5){$\mathcal{C}_0$}
\put(2,5){\vector(-2,-3){2}}
\put(3,5){\vector(2,-3){2}}
\put(9.3,7.3){$\mathcal{C}_{-1}$}
\thicklines
\put(9.5,7){\vector(-3,-4){3.7}}
\put(9,7.2){\line(-3,-2){4.5}}
\put(3,3.5){\vector(-3,-2){2.3}}
\put(9,7.3){\vector(-4,-1){5.6}} \label{oneMor}
\end{picture}
\eeq

If we can find a fifth category $\mathcal{C}_{-2}$ with a morphism to $\mathcal{C}_{-1}$ then by composition $\mathcal{C}_{-2}$ also has a morphism to $\mathcal{C}_i$, $i=1,2,0$. This gives yet another morphism between $\mathcal{C}_1$ and $\mathcal{C}_2$ which we depict as follows:
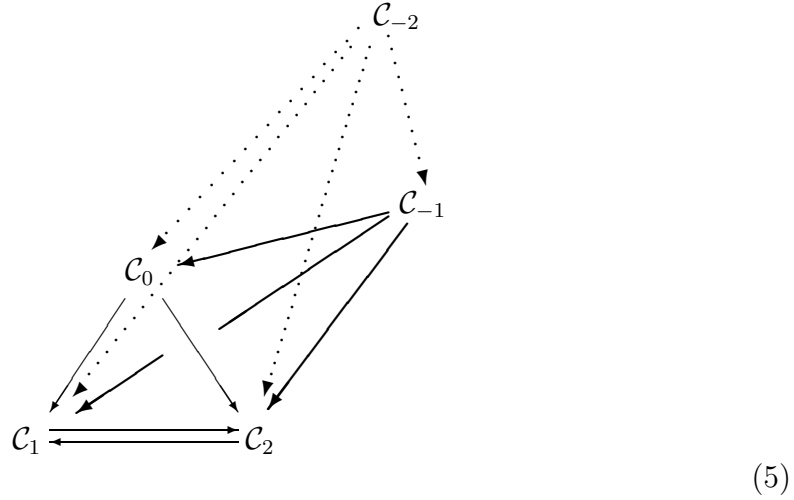
\beq
\setlength{\unitlength}{0.5cm}
\begin{picture}(12,14)(0,0)
\put(-1,1){$\mathcal{C}_1$}
\put(0,1.5){\vector(1,0){5}}
\put(5,1.2){\vector(-1,0){5}}
\put(5.2,1){$\mathcal{C}_2$}
\put(2,5.5){$\mathcal{C}_0$}
\put(2,5){\vector(-2,-3){2}}
\put(3,5){\vector(2,-3){2}}
\put(9.3,7.3){$\mathcal{C}_{-1}$}
\thicklines
\put(9.5,7){\vector(-3,-4){3.7}}
\put(9,7.2){\line(-3,-2){4.5}}
\put(3,3.5){\vector(-3,-2){2.3}}
\put(9,7.3){\vector(-4,-1){5.6}}
\put(8.6,12.3){$\mathcal{C}_{-2}$}
\multiput(9,12)(0.09,-0.36){10}{\circle*{0.1}}
\put(9.9,8.4){\vector(1,-4){0.1}}
\multiput(8.5,11.7)(-0.09,-0.3){30}{\circle*{0.1}}
\put(5.8,2.8){\vector(-1,-4){0.1}}
\multiput(8,11.7)(-0.24,-0.3){30}{\circle*{0.1}}
\put(0.7,2.5){\vector(-1,-1){0.1}}
\multiput(8.2,12.2)(-0.3,-0.32){18}{\circle*{0.1}}
\put(2.8,6.4){\vector(-1,-1){0.1}}
\end{picture} \label{twoMor}
\eeq

By induction we can construct more elaborate morphisms between $\mathcal{C}_1$ and $\mathcal{C}_2$ if given a morphism with topmost category $\mathcal{C}_{-p}$ we can find yet another one $\mathcal{C}_{-(p+1)}$ with a morphism to $\mathcal{C}_{-p}$, and by composition to all lower categories $\mathcal{C}_i$, $i=1,2,0,-1,\cdots, -(p-1), -p$ as well. If we arrange all these lower categories in a circle, we obtain a \textbf{cone} with vertex $\mathcal{C}_{-(p+1)}$. We can now define a morphism between two small categories $\mathcal{C}_1$ and $\mathcal{C}_2$ of Cat$^{\triangleq}$ to be given by an adjunction between them as well as a sequence of morphisms $\mathcal{C}_{-p} \rightarrow \mathcal{C}_{-(p-1)} \rightarrow \cdots \rightarrow \mathcal{C}_0$, morphisms in Cat, with the last category $\mathcal{C}_0$ having a morphism to both $\mathcal{C}_1$ and $\mathcal{C}_2$ as in \eqref{simplecone}. A morphism between two such objects of Cat$^{\triangleq}$ will also be referred to as a cone over $\mathcal{C}_1$ and $\mathcal{C}_2$ with vertex $\mathcal{C}_{-p}$. Observe that we have been using the terminology ``morphism between objects" as opposed to saying ``morphism from ... to ...". If $\mathcal{C}_1$ and $\mathcal{C}_2$ are two objects of Cat$^{\triangleq}$, with a morphism $\phi$ between them given in part by an adjunction $f: \mathcal{C}_1 \rightleftarrows \mathcal{C}_2:g$, then we can talk about a morphism from $\mathcal{C}_1$ to $\mathcal{C}_2$ as being $\phi$ where we regard $f$ as being a left adjoint and $g$ a right adjoint. By symmetry, we can talk about this same $\phi$ as also being a morphism from $\mathcal{C}_2$ to $\mathcal{C}_1$ with now $g$ being a left adjoint and $f$ a right adjoint. For simplicity of notation, writing $\mathcal{C} \rightarrow \mathcal{C}'$ will mean that there exists a morphism between $\mathcal{C}$ and $\mathcal{C}'$. The identity morphism is given by the natural morphism from the trivial category $\emptyset$ to an object $\mathcal{C}$ of Cat$^{\triangleq}$. Composition of morphism is defined by the product of the vertices of each cone. If we have a morphism $f$ between categories $\mathcal{C}_1$ and $\mathcal{C}_2$ given by a cone with vertex $\mathcal{C}_{-p}^{(1,2)}$ and a morphism $g$ between $\mathcal{C}_2$ and $\mathcal{C}_3$ given by a cone with vertex $\mathcal{C}_{-q}^{(2,3)}$, then the composition $g \circ f$ is given by the cone resulting from taking the product of categories $\mathcal{C}_{-p}^{(1,2)} \times \mathcal{C}_{-q}^{(2,3)}$ as this category has a projection to $\mathcal{C}_{-p}^{(1,2)}$ and $\mathcal{C}_{-q}^{(2,3)}$, and thus to all lower  categories by composition. This is represented as follows:
\beq
\setlength{\unitlength}{0.5cm}
\begin{picture}(9,10)(0,0)
\put(0,0){$\mathcal{C}_1$}
\put(4,0){$\mathcal{C}_2$}
\put(0.9,0.5){\vector(1,0){2.8}}
\put(3.7,0.2){\vector(-1,0){2.8}}
\put(8,0){$\mathcal{C}_3$}
\put(4.9,0.5){\vector(1,0){2.8}}
\put(7.7,0.2){\vector(-1,0){2.8}}
\put(1,4){$\mathcal{C}_{-p}^{(1,2)}$}
\put(1.3,3.8){\vector(-1,-3){1}}
\put(2.4,3.7){\vector(1,-2){1.5}}
\put(5,4){$\mathcal{C}_{-q}^{(2,3)}$}
\put(5.3,3.8){\vector(-1,-3){1}}
\put(6.4,3.7){\vector(1,-2){1.4}}
\put(1.5,8.8){$\mathcal{C}_{-p}^{(1,2)} \times \mathcal{C}_{-q}^{(2,3)}$}
\put(3,8){\vector(-1,-2){1.3}}
\put(4.6,8){\vector(1,-2){1.3}}
\end{picture} \label{compo}\\ \\
\eeq
which yields a morphism from $\mathcal{C}_1$ to $\mathcal{C}_3$ as depicted below:
\beq
\setlength{\unitlength}{0.5cm}
\begin{picture}(7,7)(0,0)
\put(-1,1){$\mathcal{C}_1$}
\put(0,1.5){\vector(1,0){5}}
\put(5,1.2){\vector(-1,0){5}}
\put(5.2,1){$\mathcal{C}_3$}
\put(0.5,5.5){$\mathcal{C}_{-p}^{(1,2)} \times \mathcal{C}_{-q}^{(2,3)}$}
\put(2,5){\vector(-2,-3){2}}
\put(3,5){\vector(2,-3){2}}
\end{picture}
\eeq
where for simplicity we have only displayed the vertex of the cone and all intermediate categories and morphisms are implied. This definition of composition makes the identity morphism satisfy $id \circ f = f = f\circ id$ for any morphism $f$ owing to the fact that for any category $\mathcal{C}$, $\emptyset \times \mathcal{C} \simeq \mathcal{C} \simeq \mathcal{C} \times \emptyset$. Associativity is also immediate and can be represented graphically with cones only as done below:
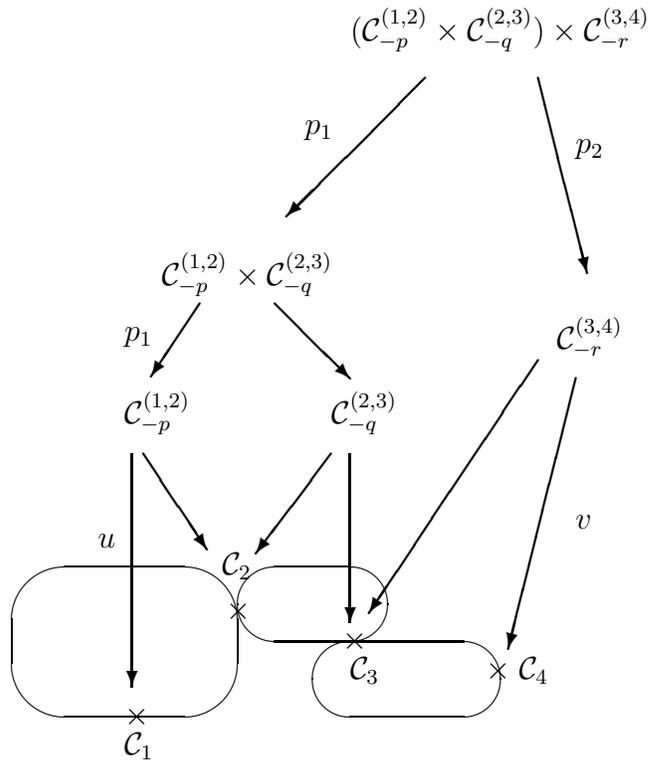
\beq
\setlength{\unitlength}{0.5cm}
\begin{picture}(20,20)(0,0)
\put(3,3){\oval(6,4)}
\put(8,4){\oval(4,2)}
\put(9,2){\oval(2,2)[l]}
\put(9,2){\oval(8,2)[r]}
\put(3,0.8){$\times$}
\put(5.7,3.6){$\times$}
\put(8.8,2.8){$\times$}
\put(12.6,2){$\times$}
\put(3,0){$\mathcal{C}_1$}
\put(5.6,4.8){$\mathcal{C}_2$}
\put(9,2){$\mathcal{C}_3$}
\put(13.5,2){$\mathcal{C}_4$}
\put(3,8.8){$\mathcal{C}_{-p}^{(1,2)}$}
\put(8.5,8.8){$\mathcal{C}_{-q}^{(2,3)}$}
\put(14.5,10.8){$\mathcal{C}_{-r}^{(3,4)}$}
\thicklines
\put(3.2,8){\vector(0,-1){6.2}}
\put(2.3,5.5){$u$}
\put(3.5,8){\vector(2,-3){1.7}}
\put(9,8){\vector(0,-1){4.5}}
\put(8.5,8){\vector(-3,-4){2}}
\put(14,10.5){\vector(-2,-3){4.5}}
\put(15,10){\vector(-1,-4){1.8}}
\put(15,6){$v$}
\put(4,12.5){$\mathcal{C}_{-p}^{(1,2)} \times \mathcal{C}_{-q}^{(2,3)}$}
\put(5,12){\vector(-2,-3){1.3}}
\put(3,11){$p_1$}
\put(7,12){\vector(1,-1){2}}
\put(9,19){$(\mathcal{C}_{-p}^{(1,2)} \times \mathcal{C}_{-q}^{(2,3)}) \times \mathcal{C}_{-r}^{(3,4)}$}
\put(11,18){\vector(-1,-1){3.7}}
\put(7.8,16.5){$p_1$}
\put(14,18){\vector(1,-4){1.3}}
\put(15,16){$p_2$}
\end{picture}
\eeq
This cone is clearly equivalent to the following cone:
\beq
\setlength{\unitlength}{0.5cm}
\begin{picture}(20,20)(0,0)
\put(3,3){\oval(6,4)}
\put(8,4){\oval(4,2)}
\put(9,2){\oval(2,2)[l]}
\put(9,2){\oval(8,2)[r]}
\put(3,0.8){$\times$}
\put(5.7,3.6){$\times$}
\put(8.8,2.8){$\times$}
\put(12.6,2){$\times$}
\put(3,0){$\mathcal{C}_1$}
\put(5.6,4.8){$\mathcal{C}_2$}
\put(9,2){$\mathcal{C}_3$}
\put(13.5,2){$\mathcal{C}_4$}
\put(3,8.8){$\mathcal{C}_{-p}^{(1,2)}$}
\put(8.5,8.8){$\mathcal{C}_{-q}^{(2,3)}$}
\put(12.5,6){$\mathcal{C}_{-r}^{(3,4)}$}
\thicklines
\put(3.2,8){\vector(0,-1){6.2}}
\put(2.3,5.5){$u$}
\put(3.5,8){\vector(2,-3){1.7}}
\put(9,8){\vector(0,-1){4.5}}
\put(8.5,8){\vector(-3,-4){2}}
\put(12.5,5.5){\vector(-3,-2){3}}
\put(13,5.5){\vector(0,-1){2.3}}
\put(13.5,4){$v$}
\put(10,14){$\mathcal{C}_{-q}^{(2,3)} \times \mathcal{C}_{-r}^{(3,4)}$}
\put(11,13){\vector(-2,-3){1.7}}
\put(13,13){\vector(0,-1){5.3}}
\put(13.5,11){$p_2$}
\put(4,18){$\mathcal{C}_{-p}^{(1,2)} \times (\mathcal{C}_{-q}^{(2,3)} \times \mathcal{C}_{-r}^{(3,4)})$}
\put(6,17){\vector(-1,-4){1.7}}
\put(4,15){$p_1$}
\put(10,17){\vector(1,-1){2}}
\put(11,16.5){$p_2$}
\end{picture}
\eeq
One interesting peculiarity of Cat$^{\triangleq}$ is that for two objects $\mathcal{C}_1$ and $\mathcal{C}_2$ of Cat$^{\triangleq}$ with a morphism $\Phi$ between them given in part by an adjunction $f \dashv g$ where $f: \mathcal{C}_1 \rightleftarrows \mathcal{C}_2 :g$, then $\Phi$ can equivalently be seen as a morphism $g: \mathcal{C}_2 \rightleftarrows \mathcal{C}_1 :f$. If $\Phi$ is seen as a morphism from $\mathcal{C}_1$ to $\mathcal{C}_2$ then $f \dashv g$ means $f$ is a left adjoint and $g$ a right adjoint. Likewise if $\Phi$ is seen as a morphism from $\mathcal{C}_2$ to $\mathcal{C}_1$, then $g \dashv f$ means $g$ is left adjoint.\\

This may lead to confusion however. Consider the following diagram:
\beq
\setlength{\unitlength}{0.5cm}
\begin{picture}(7,6)(0,-1)
\put(0,3.8){$\mathcal{C}_1$}
\put(1,4){\vector(1,0){4}}
\put(5.8,3.8){$\mathcal{C}_2$}
\put(6,3){\vector(0,-1){2.8}}
\put(5.8,-1){$\mathcal{C}_3$}
\put(1,3){\vector(4,-3){4}}
\put(2.5,4.5){$\phi$}
\put(6.5,2){$\psi$}
\put(2,1){$\zeta$}
\end{picture} \nonumber
\eeq
where $\phi$ is given in part by $f_{12}: \mathcal{C}_1 \rightleftarrows \mathcal{C}_2 :g_{12}$, $\psi$ is given in part by $f_{23}: \mathcal{C}_2 \rightleftarrows \mathcal{C}_3 : g_{23}$ and $\zeta$ is given in part by $f_{13}: \mathcal{C}_1 \rightleftarrows \mathcal{C}_3 : g_{13}$. Then $\psi$ can be viewed as a morphism $\mathcal{C}_3 \xrightarrow{\psi^{\supset}} \mathcal{C}_2$ with $g_{23} \dashv f_{23}$ and the cone for $\psi$ flipped accordingly. One may be tempted to write:
\beq
\setlength{\unitlength}{0.5cm}
\begin{picture}(7,6)(0,-1)
\put(0,3.8){$\mathcal{C}_1$}
\put(1,4){\vector(1,0){4}}
\put(5.8,3.8){$\mathcal{C}_2$}
\put(6,0.2){\vector(0,1){2.8}}
\put(5.8,-1){$\mathcal{C}_3$}
\put(1,3){\vector(4,-3){4}}
\put(2.5,4.5){$\phi$}
\put(6.5,1){$\psi^{\supset}$}
\put(2,1){$\zeta$}
\end{picture} \nonumber
\eeq
This is not possible however by definition of the composition of morphisms in Cat$^{\triangleq}$. If one were to accept a diagram such as this one then the horizontal arrow would no longer be $\phi$ but another morphism $\xi=\psi^{\supset} \circ \zeta$ with $\zeta$ itself equal to the composition $\psi \circ \phi$. This rigidity is nevertheless a good thing insofar as it enables one to keep track of previous compositions.\\

Recall that a category is said to be connected if it is inhabited and there is a finite sequence of morphisms between any two of its objects. We generalize this definition to the case of possibly infinite sequences of such morphisms and state the following conjecture:
\begin{Conj}
Cat$^{\triangleq}$ is a connected category.
\end{Conj}
We introduce a construction over Cat$^{\triangleq}$ which we anticipate to be instrumental in proving this conjecture. We first put a topology on Cat$^{\triangleq}$.

\subsection{Morphism topology on Cat$^{\triangleq}$}
\begin{Xcount}
Let $X$ denote the number of products of categories resulting from a composition in the cone associated to a morphism between any two objects of Cat$^{\triangleq}$.
\end{Xcount}
\begin{metric}
A morphism metric $d$ on Cat$^{\triangleq}$ is a function:
\begin{align}
\Big(Ob(\text{Cat}^{\triangleq}) \times Ob(\text{Cat}^{\triangleq})\Big) &\times_{\text{Mor}(\text{Cat}^{\triangleq})} \text{Arr}(\text{Cat}^{\triangleq})\rightarrow \{0\} \cup \mathbb{N} \cup \{ \infty \} \nonumber \\
(\mathcal{C}_1, \mathcal{C}_2, \phi) \mapsto d(\mathcal{C}_1, \mathcal{C}_2, \phi)&=\left\{
                                         \begin{array}{rl}
                                           0, & \text{if } \mathcal{C}_1=\mathcal{C}_2, \; \phi=id_{\mathcal{C}_1} \\
                                           X(\phi)+1, & \text{if } \phi \neq id \\
                                           \infty, & \text{if } \text{Hom}(\mathcal{C}_1, \mathcal{C}_2) = \emptyset
                                         \end{array}
                                       \right. \nonumber
\end{align}
\end{metric}
Clearly $d(\mathcal{C}_1, \mathcal{C}_2, \phi)=0 \Rightarrow \mathcal{C}_1=\mathcal{C}_2$ and $d(\mathcal{C}_1, \mathcal{C}_2, \phi)=d(\mathcal{C}_2, \mathcal{C}_1, \phi)$. Further if there are morphisms $\phi$ between $\mathcal{C}_1$ and $\mathcal{C}_2$, and $\psi$ between $\mathcal{C}_2$ and $\mathcal{C}_3$, then $d(\mathcal{C}_1, \mathcal{C}_2, \phi) + d(\mathcal{C}_2, \mathcal{C}_3, \psi) = d(\mathcal{C}_1, \mathcal{C}_3, \psi \circ \phi)$. If there is no morphism between at least one pair of categories above then the inequality $\infty + X+1 =\infty$ holds.\\

For $\mathcal{C}$ in Cat$^{\triangleq}$, $r \in \mathbb{N} \cup \{\infty \}$, we define the ball:
\beq
B(\mathcal{C}, r)=\{ \mathcal{C}' \xrightarrow{\phi} \mathcal{C} \quad | \quad d(\mathcal{C}', \mathcal{C}, \phi)<r \}  \nonumber
\eeq
Note that if $r=1$, then:
\beq
B(\mathcal{C}, 1)=\{\mathcal{C}' \xrightarrow{\phi} \mathcal{C} \quad | \quad d(\mathcal{C}', \mathcal{C}, \phi)=0 \}= \{ \mathcal{C} \circlearrowleft id_{\mathcal{C}} \} \nonumber
\eeq
and if $r=\infty$, then:
\beq
B(\mathcal{C}, \infty)=\{\text{all } \mathcal{C}' \rightarrow \mathcal{C} \} \nonumber
\eeq
\newline
The collection of all balls $B(\mathcal{C},r)$ for $\mathcal{C} \in Ob(\text{Cat}^{\triangleq})$, $r \in \mathbb{N}\cup \{\infty\}$, forms a basis for a topology on Cat$^{\triangleq}$ that we call the \textbf{morphism topology}.\\

\subsection{Crystals of categories}
In proving the connectedness of Cat$^{\triangleq}$, one may wish to prove that the geometric realization of its nerve is a connected topological space. In order to better appreciate the geometric realization of the nerve of Cat$^{\triangleq}$, we wish to consider neighborhoods of each category defined by those morphisms for which $X=0$. We are led to consider the blow up of a category at an object, which we define below:
\begin{blowup}
Let $\mathcal{C}$ be a category, $c$ an object of $\mathcal{C}$. We define the blow up of $\mathcal{C}$ at $c$ to be the category
\beq
\widetilde{\mathcal{C}}_c=\mathcal{C}-c \bigcup_{\pi}E
\eeq
where the exceptional divisor $E$ is defined to be the collection of objects of $\mathcal{C}$ that are connected to $c$ by at least one morphism, and of those morphisms between $c$ and any such object that do not result from a composition of morphisms in $\mathcal{C}$. We have a projection morphism $\pi: \tilde{\mathcal{C}}_c \rightarrow \mathcal{C}$ that is the identity morphism away from $c$ and projects $E$ down to $c$. The gluing in $\widetilde{\mathcal{C}}_c$ is done via identity morphisms for those objects that are in $E$.
\end{blowup}
In our context, this becomes:
\begin{blowup2}
We define the blow up of Cat$^{\triangleq}$ at some object $\mathcal{C}$ to be given by the category:
\beq
\widetilde{\text{Cat}^{\triangleq}}_{\mathcal{C}}=\text{Cat}^{\triangleq}-\mathcal{C} \bigcup_{\pi}B(\mathcal{C},2)
\eeq
where $B(\mathcal{C},2)$ is really $\{ \mathcal{C} \circlearrowleft id_{\mathcal{C}}\} \cup \{\mathcal{C}' \xrightarrow{\phi} \mathcal{C} \; | \; X(\phi)=0 \}$. Morphisms between any two objects of this category are defined to be the adjunctions only of morphisms between the same objects in the base category Cat$^{\triangleq}$.
\end{blowup2}
Connecting such blowups for neighboring categories in Cat$^{\triangleq}$, categories for wich there is at least one morphism between them with $X=0$, we obtain what we call a \textbf{crystal of categories}. We abbreviate the word adjunction in the definition that follows by the symbol $\dashv$.
\begin{crystal}
\beq
\widetilde{\text{Cat}^{\triangleq}}_{X=0}=\bigcup_{\substack{X=0 \:\dashv 's \\ \mathcal{C} \in \text{Cat}^{\triangleq}}}\widetilde{\text{Cat}^{\triangleq}}_{\mathcal{C}}
\eeq
\end{crystal}
Though not a category itself, this crystal of categories is a \textbf{stackoid} in the following sense: for $\mathcal{U}$ an open set about an object $\mathcal{C}$ of Cat$^{\triangleq}$, if we write $\mathcal{F}_i=\widetilde{\text{Cat}^{\triangleq}}_{\mathcal{C}_i}$, and if for any two objects $\mathcal{C}_i$ and $\mathcal{C}_j$ of $\mathcal{U}$ we have:
\beq
\mathcal{U} \cap B(\mathcal{C}_i,2) \cap B(\mathcal{C}_j,2)=
\Bigg \{\quad
\begin{CD}
\mathcal{C}_{ij;k} @>X=0>> \mathcal{C}_j \\
@VVX=0V @VVV\\
\mathcal{C}_i @>>> \mathcal{C}
\end{CD}
\text{ in } \mathcal{U}\:|\: k \in K \Bigg\}
\eeq
$K$ an indexing set. Then we have:
\beq
\mathcal{F}_i\Big|_{\mathcal{C}_{ij;k}}=\widetilde{\text{Cat}^{\triangleq}}_{\mathcal{C}_{ij;k}}
=\mathcal{F}_j\Big|_{\mathcal{C}_{ij;k}}
\eeq
for all $k \in K$. Further such a descent datum is effective since $\widetilde{\text{Cat}^{\triangleq}}_{\mathcal{C}}\Big|_{\mathcal{C}_i}=\mathcal{F}_i$. We regard $\widetilde{\text{Cat}^{\triangleq}}_{X=0}$ as being easier to deal with than the original category Cat$^{\triangleq}$ itself. Further, instead of proving the geometric realization of its nerve to be connected, we think it would be somewhat easier to work with the following object:
\begin{tangentsheaf}
The \textbf{tangent stackoid} of $\widetilde{\text{Cat}^{\triangleq}}$ is defined to be:
\beq
T\widetilde{\text{Cat}^{\triangleq}}_{X=0}:=\bigcup_{\substack{X=0 \:\dashv 's \\ \mathcal{C} \in \text{Cat}^{\triangleq}}}B(\mathcal{C},2)
\eeq
\end{tangentsheaf}
The tangent stackoid can be represented as an $\mathbb{S}^{\infty}$ sphere which we define below:
\begin{Sinfty}
Within a given formalism, an $\mathbb{S}^{\infty}$ sphere is an infinite collection of points, each of which is connected (in a sense to be precised by the context) to infinitely many other points.
\end{Sinfty}
For $T\widetilde{\text{Cat}^{\triangleq}}_{X=0}$, fix $\mathcal{C}$ an object of Cat$^{\triangleq}$, and pinch it out of $B(\mathcal{C},2)$ so that one obtains a cone-like picture with $\mathcal{C}$ as its vertex. If we glue all other balls to $B(\mathcal{C},2)$ to obtain the tangent stackoid $T\widetilde{\text{Cat}^{\triangleq}}_{X=0}$, by composition $\mathcal{C}$ is connected to infinitely many other objects of Cat$^{\infty}$, hence the $\mathbb{S}^{\infty}$ picture. We write $\mathbb{S}^{\infty}[\text{Cat}^{\triangleq}]$ for an $\mathbb{S}^{\infty}$ in this formalism.\\

For two morphisms $f$ and $g$ that can be composed as in $\mathcal{C}_1 \xrightarrow{f} \mathcal{C}_2 \xrightarrow{g} \mathcal{C}_3$, the use of $g$ follows that of $f$. This introduces a convenient notion of time. We thus regard such a composition as a two steps process. We refer to the composition of two $X=0$ morphisms as an \textbf{elementary composition}. The composition of three $X=0$ morphisms $f$, $g$ and $h$ as in $\mathcal{C}_1 \xrightarrow{f} \mathcal{C}_2 \xrightarrow{g} \mathcal{C}_3 \xrightarrow{h} \mathcal{C}_4$ is seen as the overlap of the following two elementary compositions: $\mathcal{C}_1 \xrightarrow{f} \mathcal{C}_2 \xrightarrow{g} \mathcal{C}_3$ and $\mathcal{C}_2 \xrightarrow{g} \mathcal{C}_3 \xrightarrow{h} \mathcal{C}_4$, the overlap happening over $g$. If there are morphisms between objects of a given category, to say we use one of those morphisms means we consider the application of such a morphism on a given source to obtain the corresponding target. To say that overlapping elementary compositions are \textbf{used quasi-simultaneously} means that if $\mathcal{C}_1 \xrightarrow{f} \mathcal{C}_2 \xrightarrow{g} \mathcal{C}_3$ and $\mathcal{C}_2 \xrightarrow{g} \mathcal{C}_3 \xrightarrow{h} \mathcal{C}_4$ are two such elementary compositions, once in the first composition $g$ is being used, the second composition is being used also, starting by $g$ of course. If there is a sequence of elementary, overlapping compositions that is used quasi-simultaneously in some category, then we say we have a \textbf{flow} along such a sequence. If such a sequence is closed in the sense that such a sequence does not have a final object, for example if it is cyclic, then the flow is said to be \textbf{conservative}. In the case of $\text{Cat}^{\triangleq}$ if we have a flow along all closed sequences of morphisms, and if in addition they can all be used simultaneously in a compatible manner, then we say we have a flow on $\mathbb{S}^{\infty}[\text{Cat}^{\triangleq}]$.
\begin{SinftyConj}
On $\mathbb{S}^{\infty}[\text{Cat}^{\triangleq}]$ there is at least one flow. Further if there is one flow, there is infinitely many others.
\end{SinftyConj}
%Studying minimal deformations of a given flow represented by a family of neighborhing flows related to one another by minimal variations, or studying the behavior of such families under reversal of compositions along closed sequences of morphisms (something we can do because morphisms in Cat$^{\infty}$ are adjunctions) should shed some light on the geometry of $\mathbb{S}^{\infty}[\text{Cat}^{\infty}]$, and therefore on whether Cat$^{\infty}$ is connected.\\

%A proof of connectedness seems at present out of reach. In spite of the fact that we still believe the above conjectures to be true, we nevertheless present in the next subsection another model along the same lines.\\

\subsection{Fractal $\infty$-categories}
In this subsection we do away with the axiom of regularity from Zermelo-Frankel set theory (\cite{F}, \cite{Z1}, \cite{Z2}) one of whose consequences is that no set can be an element of itself since that is precisely what we aim to do in what follows. An alternative to axiomatic set theory is found in non-well founded set theories \cite{Mi}. We do not regard the ZF axioms or those of non-well founded set theories as mutually exclusive but rather as complementary tools whose use is contextually warranted. Those axioms of non-well founded set theories that are in contradiction with the axiom of regularity are called anti-foundation axioms. We can mention Boffa's work in this regard ( \cite{Bo1}, \cite{Bo2}). We place ourselves within Boffa's universe \textbf{B}. What follows should read as a research announcement.\\

We consider pointed $\mathbb{Z}$-graded category-like objects that, for lack of a better name we will simply refer to as \"uber-categories for reasons that will become clear shortly. If we consider objects and morphisms only, an \"uber-category is nothing but an ordinary category for which there is only one morphism from one object to another. From that perspective, \"uber-categories are constructed from ordinary categories following a ``minimal" splitting principle; as soon as we introduce higher morphisms, a distinction is being made in that if there exists two \textbf{elementary} morphisms $f,g: a \rightarrow b$ in a given category $\mathcal{C}$ (elementary as in not a composition of other morphisms), $a$ and $b$ being objects of $\mathcal{C}$, and if there exists a 2-morphism $\tau: f \Rightarrow g$, then we split $\mathcal{C}$ into two categories, one consisting of $f:a \rightarrow b$, the other of $g:a \rightarrow b$, each resulting category being an \"uber-category in disguised form.\\

Rather than obfuscate the reader with bewildering layers of definitions, we will now delve into \"uber-categories proper rather informally, in the hope that this will make it clear what an \"uber-category is, perhaps to the detriment of more sophistication in the presentation, something we do not see as being entirely necessary for the time being. If we regard the two categories above as elements of a larger category with $\tau$ as a map from one to the other, these two categories become \textbf{\"uber-categories} properly speaking in the sense that once seen as elements of a larger category, their morphisms become objects of the larger category, and higher morphisms such as $\tau$ become morphisms of the larger category. In other terms, isolated \"uber-categories are simply ordinary categories with objects and a single morphism between them. Embedded within larger categories, their morphisms become objects of the larger category. Identity morphisms and compositions are defined as in higher category theory, and associativity is defined within the context of weak $\infty$-categories at the very least, but preferably with $(\infty,1)$-categories in mind.\\

By induction, categories with pairs of objects between which there are multiple elementary morphisms as well as at least one higher morphism from one of those morphisms to another, will split as above into \"uber-categories, each one displaying only one elementary morphism from the category it's originating from. Regarding notation, each \"uber-category resulting from such a split will be labeled by its morphisms such as in $\mathcal{C}^{\{f_i\}_{i \in I}}$, $I$ an indexing set for the set of morphisms in such a category.\\

Concerning morphisms, given a $n$-morphism, $n+1$-morphisms are morphism objects of possibly yet another category, so we place ourselves within the context of enriched category theory, with enrichments taking place at each stage.\\

For $n \in \mathbb{Z}$ fixed, $\mathcal{C}^{(*)}$ one pointed $\mathbb{Z}$-graded \"uber-category, $\mathcal{C}^{(n)}$ is pointed at $\mathcal{C}^{(n-1)}$, one of its elements. Its other elements are all those ($n$-1)-st categories $\mathcal{C}^{(n-1)}$ have morphisms with, a morphism $\phi^{(n-1)}$ between $\mathcal{C}^{(n-1)}$ and another category $\mathcal{D}^{(n-1)}$ being defined to be an adjunction $f_{\phi^{(n-1)}}:\mathcal{C}^{(n-1)} \rightleftarrows \mathcal{D}^{(n-1)}:g_{\phi^{(n-1)}}$, also denoted by $f_{\phi^{(n-1)}} \dashv g_{\phi^{(n-1)}}$. We do not label categories by their morphisms for the time being, as they should, but by $(n-1)$ instead to emphasize at which stage we are working. Suppose we have an elementary morphism $\mu^{(n-1)}$ between two objects $\mathcal{E}^{(n-1)}$ and $\mathcal{D}^{(n-1)}$ of $\mathcal{C}^{(n)}$. If in addition to $\mu^{(n-1)}$ there are other elementary morphisms between $\mathcal{E}^{(n-1)}$ and $\mathcal{D}^{(n-1)}$, then we split $\mathcal{C}^{(n)}$ into as many copies as there are other elementary morphisms between $\mathcal{E}^{(n-1)}$ and $\mathcal{D}^{(n-1)}$. For $\mu^{(n-1)}$ and $\psi^{(n-1)}$ two such morphisms, we have two copies, one that we denote by $\mathcal{C}^{\{\cdots, \mu^{(n-1)}, \cdots \}, (n)}$, the other $\mathcal{C}^{\{\cdots, \psi^{(n-1)}, \cdots \}, (n)}$, with a morphism $\xi^{(n)}$ between them given by an adjunction:
\beq
f_{\xi^{(n)}}:\mathcal{C}^{\{\cdots, \mu^{(n-1)}, \cdots \}, (n)}  \rightleftarrows \mathcal{C}^{\{\cdots, \psi^{(n-1)}, \cdots \}, (n)}:g_{\xi^{(n)}}
\eeq
If such an adjunction is between $\mathcal{C}^{\{\cdots, \mu^{(n-1)}, \cdots \}, (n)}$ and $\mathcal{C}^{\{\cdots, \psi^{(n-1)}, \cdots \}, (n)}$, viewed as objects of some larger category $\mathcal{C}^{(n+1)}$, then $f_{\xi^{(n)}}$ is an object of \\ Hom$_{\mathcal{C}^{(n+1)}}(f_{\mu^{(n-1)}},f_{\psi^{(n-1)}})$, and likewise $g_{\xi^{(n)}}$ is an object of \\ Hom$_{\mathcal{C}^{(n+1)}}(g_{\psi^{(n-1)}},g_{\mu^{(n-1)}})$. \\

Imposing that elementary morphisms between same objects split as opposed to non-elementary ones hold our categories together. Indeed, for a category $\mathcal{C}^{(n)}$ pointed at $\mathcal{C}^{(n-1)}$, suppose we have the following morphisms:
\beq
\setlength{\unitlength}{0.5cm}
\begin{picture}(4,7)(0,-1)
\put(1,-0.5){$\mathcal{C}^{(n-1)}$}
\put(0.8,0.8){$\times$}
\put(3,4){$\times$}
\put(4,4){$\mathcal{D}^{(n-1)}$}
\put(-0.5,3){$\times$}
\put(-3,3){$\mathcal{E}^{(n-1)}$}
\thicklines
\put(1,1){\vector(2,3){2}}
\put(2.5,2){$\psi$}
\put(1,1){\vector(-1,2){1}}
\put(-0.5,2){$\phi$}
\end{picture}
\eeq
Then we have by composition (we work with adjunctions) a morphism from $\mathcal{E}^{(n-1)}$ to $\mathcal{D}^{(n-1)}$ that makes this diagram commutative as in:
\beq
\setlength{\unitlength}{0.5cm}
\begin{picture}(4,7)(0,-1)
\put(1,-0.5){$\mathcal{C}^{(n-1)}$}
\put(0.8,0.8){$\times$}
\put(3,4){$\times$}
\put(4,4){$\mathcal{D}^{(n-1)}$}
\put(-0.5,3){$\times$}
\put(-3,3){$\mathcal{E}^{(n-1)}$}
\thicklines
\put(1,1){\vector(2,3){2}}
\put(2.5,2){$\psi$}
\put(1,1){\vector(-1,2){1}}
\put(-0.5,2){$\phi$}
\multiput(0,3.2)(0.45,0.15){8}{\circle*{0.15}}
\put(2.7,4.1){\vector(3,1){0.8}}
\put(1,4){$\xi$}
\end{picture}
\eeq
If we did not limit ourselves to elementary morphisms, then we would have to split $\psi$ and $\xi \circ \phi$, which would lead to a splitting of all morphisms, and our categories would end up not being very interesting.\\

Another related question is the following. Given two categories $\mathcal{C}^{(n)}$ and $\mathcal{D}^{(n)}$, $\mathcal{C}^{(q)}$ pointed at $\mathcal{C}^{(q-1)}$, $\mathcal{D}^{(q)}$ pointed at $\mathcal{D}^{(q-1)}$, such that $\mathcal{C}^{(p)}$ is never an element of $\mathcal{D}^{(p+1)}$ for $p<n$, can we have a morphism between $\mathcal{C}^{(n)}$ and $\mathcal{D}^{(n)}$? Such a phenomenon we refer to as the \textbf{sudden appearance} of a morphism. Recall that a category $\mathcal{C}$ not being an element of another category $\mathcal{D}$ means that there is no morphism between $\mathcal{C}$ and the base category of $\mathcal{D}$.
\begin{sudden}
There are no sudden appearances of morphisms for pointed $\mathbb{Z}$-graded \"uber-categories as defined above.
\end{sudden}
\begin{proof}
Let $\mathcal{C}^{(n)}$ and $\mathcal{D}^{(n)}$ be two such categories with a morphism $\tau$ between them. Then by construction this morphism must be a higher morphism from some $\mathcal{C}_1^{(n-1)} \xrightarrow{f} \mathcal{C}_2^{(n-1)}$ in $\mathcal{C}^{(n)}$ to some $\mathcal{D}_1^{(n-1)} \xrightarrow{g} \mathcal{D}_2^{(n-1)}$ in $\mathcal{D}^{(n)}$, where both $f$ and $g$ are elementary morphisms. Further, we must have $\mathcal{C}_i^{(n-1)}=\mathcal{D}_i^{(n-1)}$, $i=1,2$. We then have a morphism from the base point $\mathcal{C}^{(n-1)}$ of $\mathcal{C}^{(n)}$ to $\mathcal{C}_1^{(n-1)}$, composed with $f$, followed by a morphism from $\mathcal{D}_2^{(n-1)}=\mathcal{C}_2^{(n-1)}$ to the base point $\mathcal{D}^{(n-1)}$ of $\mathcal{D}^{(n)}$, hence a morphism from $\mathcal{C}^{(n-1)}$ to $\mathcal{D}^{(n-1)}$. This means $\mathcal{C}^{(n-1)}$ must be an element of $\mathcal{D}^{(n)}$, a contradiction.
\end{proof}

Pointed $\mathbb{Z}$-graded \"uber-categories being constructed inductively, we could very well have considered sequences of pointed \"uber-categories instead and have worked with those objects. We will see how that approach is potentially more amenable to obtaining a result about the connectedness of the $\infty$-category of such graded \"uber-categories below. We are more interested in $\mathbb{Z}$-graged \"uber-categories insofar as from a given rank $n \in \mathbb{Z}$ we can study lower morphisms between lower categories. This amounts to studying the causality of given morphisms at the $n$-th rank, with a tacit understanding that original objects at the $p$-th rank for $p$ negative and very large may not necessarily be known. The problem however is that coherence conditions for the associativity of morphisms make this formalism intractable computation-wise.\\

Another approach is provided by considering sequences of categories instead. The problem with that approach however is that a collection of categories $\{\mathcal{C}^{(0)}\}$ is known at the onset, and causality becomes a non-issue. We regard this formalism as a diminished variant of the above theory. Our view is that from a causal perspective, once we go back further for smaller and smaller values of the rank at which we are studying \"uber-categories and morphisms between them, there is a point beyond which we can still refer to the considered objects as \"uber-categories though we may not necessarily currently have the mathematical machinery required to fully describe the objects thus obtained. This is why working with sequences of pointed \"uber-categories is unrealistic in our opinion. However, that picture carries one very important simplification. Precisely because we work with sequences and higher categories and morphisms are defined and constructed inductively, it is possible to limit ourselves to $(\infty,1)$-\"uber-categories. For that purpose, we define an adjunction $f \dashv g$ to be invertible if both $f$ and $g$ are invertible, and we take the inverse of $f \dashv g$ to be $f^{-1} \dashv g^{-1}$. We define $\mathcal{F}\text{Cat}_{\infty}$ to be the $\infty$-category of $\mathbb{Z}$-graded pointed \"uber-categories and higher morphisms between them which we refer to as the category of \textbf{fractal $\infty$-categories}. The category of $(\infty,1)$-sequences of pointed \"uber-categories we call the category of \textbf{fractal $\infty$-categories in positive degrees} which we denote by $\mathcal{F}_{\geq 0}\text{Cat}_{\infty}$.\\

For questions of connectedness, we limit ourselves to sequences of \"uber-categories. Observe that at a given rank, categories $\mathcal{C}^{(n)}$ have a definite knowledge of where they are coming from as $\mathcal{C}^{(k)}$ is known for $k < n$, but from the perspective of $\mathcal{C}^{(n)}$, nothing is known about $\mathcal{C}^{(l)}$ for $l >n$. This makes proving a connectedness statement about $\mathcal{F}_{\geq 0}\text{Cat}_{\infty}$ difficult. We introduce a notion of compactification that may be useful in that regard. Assuming such a notion as $f^{(\infty)}$ to make sense, if for all morphisms we have $f^{(0)}$ to be isomorphic to $f^{(+\infty)}$ then we can define a first compactification of $\mathcal{F}_{\geq 0}\text{Cat}_{\infty}$ to be defined by:

\beq
\mathbb{K}_1\mathcal{F}_{\geq 0}\text{Cat}_{\infty}:=\mathcal{F}_{\geq 0}\text{Cat}_{\infty}/
\text{Ob}^{(0)} \sim \text{Ob}^{(+\infty)}
\eeq
where objects at the $0$-th and $+\infty$ ranks are identified. Another compactification is given by a one point compactification if for each object $\mathcal{C}^{(*)}$ of $\mathcal{F}_{\geq 0}\text{Cat}_{\infty}$ we can find a category $\mathcal{C}^{(-)}$ such that it has $\mathcal{C}^{(+\infty)}$ as an object and for any two elementary morphisms $\phi^{(+\infty)}$ and $\psi^{(+\infty)}$ between two \"uber-categories $\mathcal{C}^{(+\infty)}$ and $\mathcal{D}^{(+\infty)}$ respectively given by  $f_{\phi^{(+\infty)}}:\mathcal{C}^{(+\infty)} \rightleftarrows \mathcal{D}^{(+\infty)}:g_{\phi^{(+\infty)}}$ and $f_{\psi^{(+\infty)}}:\mathcal{C}^{(+\infty)} \rightleftarrows \mathcal{D}^{(+\infty)}:g_{\psi^{(+\infty)}}$, we have a morphism $\mu^{(-)}$ given by $f_{\mu^{(-)}}:\mathcal{C}^{(-)} \rightleftarrows \mathcal{D}^{(-)}:g_{\mu^{(-)}}$ such that $f_{\mu^{(-)}}$ is a morphism between $f_{\phi^{(+\infty)}}$ and $f_{\psi^{(+\infty)}}$ and likewise $g_{\mu^{(-)}}$ is a morphism between $g_{\psi^{(+\infty)}}$ and $g_{\phi^{(+\infty)}}$. Suppose there is another morphism $\rho^{(-)}$ given by $f_{\rho^{(-)}}:\mathcal{C}^{(-)} \rightleftarrows \mathcal{D}^{(-)}: g_{\rho^{(-)}}$. We also need that both $\mathcal{C}^{(-)}$ and $\mathcal{D}^{(-)}$ be objects of $\mathcal{C}^{(0)}$ and $\mathcal{D}^{(0)}$, for which one of the morphisms between them is given by an adjunction $f^{(0)} \dashv g^{(0)}$, $f^{(0)}$ a morphism from $f_{\mu^{(-)}}$ to $f_{\rho^{(-)}}$. In that case we can define the one point compactification of $\mathcal{F}_{\geq 0}\text{Cat}_{\infty}$ as:
\beq
\mathbb{K}_2\mathcal{F}_{\geq 0}\text{Cat}_{\infty}:=\mathcal{F}_{\geq 0}\text{Cat}_{\infty} \coprod_{\mathcal{C} \in \mathcal{F}_{\geq 0}\text{Cat}_{\infty}}\{\mathcal{C}^{(-)}, \mathcal{C}^{(-)} \rightleftarrows \mathcal{D}^{(-)}\}
\eeq
Either compactification still yields a fractal $\infty$-category. Suppose that in either case we obtain a disconnected fractal $\infty$-category. In that case for two disconnected components we can define a twist by an integer number $p$ by letting all categories and morphisms of either component be shifted by $p$: $\mathcal{C}^{(n)} \mapsto \mathcal{C}^{(n+p)}$ and $f^{(n)} \mapsto f^{(n+p)}$. If for all $p \in \mathbb{Z}$ the corresponding components are still disconnected then we say the original category $\mathcal{F}_{\geq 0}\text{Cat}_{\infty}$ is \textbf{totally disconnected}.\\
\begin{FCatConn}
$\mathcal{F}_{\geq 0}\text{Cat}_{\infty}$ is not totally disconnected.
\end{FCatConn}
This category offers a dynamic alternative to Cat$^{\triangleq}$. Though we conjecture this latter to be connected, we still regard $\mathcal{F}_{\geq 0}\text{Cat}_{\infty}$ as being of some relevance for answering questions pertaining to causality in higher category theory, though as we said above it is really $\mathcal{F}\text{Cat}_{\infty}$ that is important. Further we anticipate such a concept as fractal $\infty$-categories to be versatile and to have applications in other branches of Mathematics.

\subsection{Grothendieck topology on $\mathcal{F}\text{Cat}_{\infty}$}

We put a $\mathbb{Z}$-graded Grothendieck topology on $\mathcal{F}\text{Cat}_{\infty}$. In doing so we will mainly follow \cite{GM} as this is a standard reference. A first observation is that since we work with enrichments of categories at each rank, it is natural to have Grothendieck topologies for each rank, hence the notion of a $\mathbb{Z}$-graded Grothendieck topology. Fix $n \in \mathbb{Z}$. We work within the $n$-th slice for the time being.\\

For $\mathcal{C}^{(*)}$ an object of $\mathcal{F}\text{Cat}_{\infty}$, we define the following family of morphisms:
\beq
\Phi^{(n)}(\mathcal{C}^{(n)})=\{\text{all morphisms } \phi_i:\mathcal{C}_i^{(n)} \rightarrow \mathcal{C}^{(n)} \}
\eeq

\begin{GrothTop}
The families of sieves $\{\Phi^{(n)}(\mathcal{C}^{(n)})\}_{\mathcal{C} \in \mathcal{F}\text{Cat}_{\infty}}$ define \\a Grothendieck topology on the $n$-th slice of $\mathcal{F}\text{Cat}_{\infty}$.
\end{GrothTop}
\begin{proof}
By definition, the set of all morphisms to $\mathcal{C}^{(n)}$ in the $n$-th slice of $\mathcal{F}\text{Cat}_{\infty}$ is precisely $\Phi^{(n)}(\mathcal{C}^{(n)})$, a covering sieve. Let $\Phi^{(n)}(\mathcal{C}^{(n)})$ be a covering sieve. Let $\mathcal{D}^{(n)} \xrightarrow{\psi^{(n)}} \mathcal{C}^{(n)}$ be a morphism in the $n$-th slice of $\mathcal{F}\text{Cat}_{\infty}$. We define the restriction $\Phi_{\mathcal{D}^{(n)}}^{(n)}(\mathcal{C}^{(n)})$ of this sieve to $\mathcal{D}^{(n)}$ to be the family of morphisms $\mathcal{C}_i^{(n)} \xrightarrow{\phi_i^{(n)}} \mathcal{D}^{(n)} \xrightarrow{\psi^{(n)}} \mathcal{C}^{(n)}$ such that the composition $\psi^{(n)} \circ \phi_i^{(n)}$ is in $\Phi^{(n)}(\mathcal{C}^{(n)})$. By definition however this is true for all $\phi_i^{(n)}$ so $\Phi_{\mathcal{D}^{(n)}}^{(n)}(\mathcal{C}^{(n)})=\Phi^{(n)}(\mathcal{D}^{(n)})$. Thus the restriction of a covering sieve is a covering sieve. Finally suppose $\Phi^{(n)}(\mathcal{C}^{(n)})=\{\phi_i^{(n)}:\mathcal{C}_i^{(n)} \rightarrow \mathcal{C}^{(n)}\}$ is a covering sieve, $\Psi^{(n)}$ is another sieve over $\mathcal{C}^{(n)}$ such that the restriction $\Psi_{\mathcal{C}_i^{(n)}}^{(n)}$ to any element $\phi_i^{(n)}: \mathcal{C}_i^{(n)} \rightarrow \mathcal{C}^{(n)}$ is a covering sieve. We write $\Psi_{\mathcal{C}_i^{(n)}}^{(n)}=\{\mathcal{D}_j^{(n)} \xrightarrow{\psi_{ji}^{(n)}} \mathcal{C}_i^{(n)} \}$. This being a covering sieve means that all morphisms from $\mathcal{D}_j^{(n)}$ to $\mathcal{C}_i^{(n)}$ constitute this family. If morphisms $\phi_i^{(n)}$ from $\mathcal{C}_i^{(n)}$ to $\mathcal{C}^{(n)}$ are given in part by an adjunction $f_i^{(n)} \dashv g_i^{(n)}$ then the same morphism between $\mathcal{C}_i^{(n)}$ and $\mathcal{C}^{(n)}$ can be viewed as a morphism $\mathcal{C}^{(n)} \rightarrow \mathcal{C}_i^{(n)}$ given in part by an adjunction $g_i^{(n)} \dashv f_i^{(n)}$, which we write as $\psi_i^{(n)\supset}$. Then $\Psi_{\mathcal{C}_i^{(n)}}^{(n)}$ being a covering sieve means that it contains in particular all morphisms $\mathcal{D}_j^{(n)} \rightarrow \mathcal{C}^{(n)} \xrightarrow{\psi_i^{(n)\supset}} \mathcal{C}_i^{(n)}$. This means in such a sequence all morphisms $\mathcal{D}_j^{(n)} \rightarrow \mathcal{C}^{(n)}$ are present. By definition of the restriction of a sieve all such sequences composed with $\phi_i^{(n)}:\mathcal{C}_i^{(n)} \rightarrow \mathcal{C}^{(n)}$ are in $\Psi^{(n)}$. Such compositions are of the form $\mathcal{D}_j^{(n)} \rightarrow \mathcal{C}^{(n)} \xrightarrow{\phi^{(n) \supset}} \mathcal{C}_i^{(n)} \xrightarrow{\phi^{(n)}} \mathcal{C}_i^{(n)}$ where we have a double composition of adjunctions, isomorphic to the identity morphism, so all morphisms $\mathcal{D}_j^{(n)} \rightarrow \mathcal{C}^{(n)}$ are in $\Psi$, hence this latter is a covering sieve.
\end{proof}
\begin{ZGrothn}
We denote by $\Phi^{(n)}$ the Grothendieck topology on the $n$-th slice of $\mathcal{F}\text{Cat}_{\infty}$.
\end{ZGrothn}
\begin{ZGroth}
A $\mathbb{Z}$-graded Grothendieck topology on an $\infty$-category of $\mathbb{Z}$-graded categories $\mathcal{C}^{(*)}$ is a collection $\{G^{(n)}\}_{n \in \mathbb{Z}}$ where each $G^{(n)}$ is a Grothendieck topology on the $n$-th slice of such a category.
\end{ZGroth}
\begin{PropGroth}
$\Phi:=\{\Phi^{(n)}\}_{n \in \mathbb{Z}}$ defines a $\mathbb{Z}$-graded Grothendieck topology on $\mathcal{F}\text{Cat}_{\infty}$ making $(\mathcal{F}\text{Cat}_{\infty}, \Phi)$ into a site.
\end{PropGroth}
Once that is achieved, one can study $\infty$-topos and stacks over this site. Whether it be Cat$^{\triangleq}$ or $\mathcal{F}\text{Cat}_{\infty}$, one needs to resort to "global" arguments, and the study of topos we expect to be illuminating in that respect.

\newpage

\appendix
\section*{\centering Appendix}

We follow the notations used by Wallbridge as much as possible. Everything below can be found in his paper \cite{W} and we have only selected that which is necessary for our purposes. The material of interest to us we have repackaged for ease of use. Another reference we also use is \cite{L1}.\\

\section{$\infty$-categories}
The higher generalization of classical category theory is fertile ground for proving myriad problems in several seemingly disconnected mathematical fields. The origins of such work can be found in \cite{BV} and \cite{J2} with very nice reviews provided in \cite{L1} and \cite{B}.\\ 

A $(n,m)$-category, $m \leq n$, is a category in which all $k$-morphisms are invertible if $m<k \leq n$. If one sets $n=\infty$ and $m=1$ one deals with $(\infty,1)$-categories, which one typically refers to as $\infty$-categories.
\begin{InftyCat} (\cite{L1}, \cite{BV})
An $\infty$-category is a simplicial set $K$ with the following extension property: for any $0<i<n$ any map $\Lambda_i^n \rightarrow K$ from the $i$-th horn of $\Delta^n$ admits an extension $\Delta^n \rightarrow K$.
\end{InftyCat}
An object of an $\infty$-category that is both initial and final is called a \textbf{zero object} and is denoted by $0$. An $\infty$-category is said to be \textbf{pointed} if it contains a zero object. A pointed $\infty$-category with finite limits and colimits for which pullback and pushout squares are the same is called a \textbf{stable} $\infty$-category.\\ %The full subcategory of Cat$_{\infty}$ spanned by stable $\infty$-categories is denoted by Cat$^{\bot}_{\infty}$.\\

The theory of $\infty$-categories is equivalent to the theory of \textbf{topological categories} (\cite{L1}). Those are categories enriched over the category $\mathcal{CG}$ of compactly generated, weakly Hausdorff topological spaces. The equivalence exists by virtue of Theorem 1.1.5.13 of \cite{L1}. The theory of $\infty$-categories is also equivalent to the theory of \textbf{simplicial categories} (\cite{L1}). These are categories enriched over the category $\mathcal{S}\text{et}_{\Delta}$ of simplicial sets. \\

One important concept is that of the homotopy category h$\mathcal{C}$ of an $\infty$-category $\mathcal{C}$. There are two equivalent definitions of h$\mathcal{C}$. One can first define a notion of homotopic maps as follows:
two edges $\phi: X \rightarrow Y$ and $\phi':X \rightarrow Y$ in an $\infty$-category $\mathcal{C}$ are said to be \textbf{homotopic} if there is a 2-simplex $\sigma: \Delta^2 \rightarrow \mathcal{C}$:
\beq
\setlength{\unitlength}{0.5cm}
\begin{picture}(7,5)(0,0)
\put(0,0.3){$X$}
\put(6.2,0.3){$Y$}
\put(3.5,3.5){$Y$}
\put(1,0.5){\vector(1,0){5}}
\put(1,1){\vector(1,1){2}}
\put(6,1){\vector(-1,1){2}}
\put(3.5,-0.5){$\phi'$}
\put(1.5,2.5){$\phi$}
\put(5.5,2.5){$id_Y$}
\end{picture} \nonumber
\eeq
\newline
The relation of homotopy is an equivalence relation on the edges between same vertices of $\mathcal{C}$ (\cite{L1}). One can define the homotopy category h$\mathcal{C}$ of an $\infty$-category $\mathcal{C}$ to be the category whose vertices are those of $\mathcal{C}$ and whose edges are given by the sets of homotopy classes of morphisms between objects of $\mathcal{C}$. Or going back to the definition of an $(\infty,1)$-category, one can define the homotopy category of a simplicial set, for which we need the definition of the homotopy category of a simplicial category, which in turn necessitates to define the homotopy category of a topological category. We define these in turns.\\

One defines the homotopy category h$\mathcal{C}$ of a topological category $\mathcal{C}$ by taking as objects those of $\mathcal{C}$, and for two objects $X$, $Y$ of $\mathcal{C}$, one defines:
\beq
\text{Hom}_{\text{h}\mathcal{C}}(X,Y)=\pi_0 \text{Map}_{\mathcal{C}}(X,Y) \nonumber
\eeq
For instance if one considers the topological category $\mathcal{C}$ whose objects are CW complexes and for objects $X$, $Y$ of such a category $\text{Map}_{\mathcal{C}}(X,Y)$ is the set of continuous maps from $X$ to $Y$, then one refers to $\mathcal{H}=\text{h}\mathcal{C}$ as the \text{homotopy category of spaces}.\\

Observe as is done in \cite{L1} that for $X \in \mathcal{CG}$, the map that assigns to $X$ its homotopy class $[X]$ gives a well-defined map $\theta: \mathcal{CG} \rightarrow \mathcal{H}$. From there one can improve on the definition of the homotopy category of a topological category by defining $\text{Hom}_{\text{h}\mathcal{C}}(X,Y)=[ \text{Map}_{\mathcal{C}}(X,Y) ]$ instead of using path components.\\

We can define the homotopy category of a simplicial category as done in \cite{L1} by applying a certain functor to each of the morphism spaces of this category. Recall that one has a geometric realization functor:
\beq
\mathcal{S}\text{et}_{\Delta} \xrightarrow{||} \mathcal{CG} \nonumber
\eeq
If $\mathcal{C}$ is a simplicial category, one defines its homotopy category h$\mathcal{C}$ as being the $\mathcal{H}$-enriched category following \cite{L1} by applying the functor $\theta \circ ||: \mathcal{S}\text{et}_{\Delta} \rightarrow \mathcal{H}$ to each of its morphism spaces. \\

Finally, in order to define the homotopy category of a simplicial set, one first has to construct some simplicial category $\mathfrak{C}[S]$ from a given simplicial set $S$. This is not very illuminating, and the interested reader is refered to \cite{L1} for details. Then one just defines the homotopy category h$S$ of a simplicial set $S$ to be the homotopy category h$\mathfrak{C}[S]$.\\

Now if $\mathcal{C}$ and $\mathcal{D}$ are $\infty$-categories, one simply defines a functor from $\mathcal{C}$ to $\mathcal{D}$ to be a map $\mathcal{C} \rightarrow \mathcal{D}$ of simplicial sets.\\

A functor $F:\mathcal{C} \rightarrow \mathcal{D}$ between two $\infty$-categories $\mathcal{C}$ and $\mathcal{D}$ is said to be \textbf{fully faithful} if its underlying maps of simplicial sets is so. If $F$ is regarded as a map between simplicial sets, it is said to be fully faithful if the induced map of homotopy categories $\text{h}F: \text{h}\mathcal{C} \rightarrow \text{h}\mathcal{D}$ is such that for all $X$, $Y$ in $\mathcal{C}$, the induced map $\text{Map}_{\text{h}\mathcal{C}}(X,Y) \rightarrow \text{Map}_{\text{h}\mathcal{D}}(FX, FY)$ is an isomorphism in $\mathcal{H}$.\\

One way to also produce $\infty$-categories is to use \textbf{model categories}, which we now define:
\begin{modelCat}(\cite{L1}, \cite{H})
A model category is a category $\mathcal{C}$ equipped with three distinguished classes of morphisms: fibrations, cofibrations and weak equivalences. In addition, such a category satisfies the following axioms:
\begin{itemize}
\item[-] $\mathcal{C}$ admits small limits and colimits.
\item[-] For any composition $X \xrightarrow{f} Y \xrightarrow{g} Z$, if any two of $g \circ f$, $f$ and $g$ are weak equivalences, so is the third.
\item[-]For a commutative diagram
\beq
\begin{CD}
X @>>i> X' @>>r> X\\
@VfVV @VVgV @VVfV \\
Y @ >>i'> Y' @>>r'> Y
\end{CD} \nonumber
\eeq
for which $r \circ i=id_X$ and $r'\circ i'=id_Y$, then if $g$ is any of a fibration, cofibration, or a weak equivalence, then so is $f$.
\item[-]Given a diagram of solid arrows
\beq
\setlength{\unitlength}{0.5cm}
\begin{picture}(6,5)(0,0)
\put(0,0){$Y$}
\put(4.5,0){$Y'$}
\put(0,4){$X$}
\put(4.5,4){$X'$}
\put(1.2,0.2){\vector(1,0){3}}
\put(0.2,3.8){\vector(0,-1){3}}
\put(-0.5,2){$i$}
\put(1.2,4.3){\vector(1,0){3}}
\put(4.7,3.8){\vector(0,-1){3}}
\put(5,2){$p$}
\multiput(1,1)(0.2,0.2){15}{\circle*{0.1}}
\put(3.5,3.5){\vector(1,1){0.5}}
\end{picture}  \nonumber \\ \\
\eeq
\\
one can find a diagonal arrow from $Y$ to $X'$ as above making the above diagram commutative if either $i$ is a cofibration and $p$ is both a fibration and a weak equivalence, or $i$ is both a cofibration and a weak equivalence and $p$ is a fibration.
\item[-]Any map $X \rightarrow Y$ admits factorizations $X \xrightarrow{i} Z \xrightarrow{p} Y$ ($i$ cofibration, $p$ fibration and weak equivalence) and $X \xrightarrow{i'} Z' \xrightarrow{p'} Y$ ($i'$ cofibration and weak equivalence, $p'$ fibration).
\end{itemize}
\end{modelCat}

For instance the category $\mathcal{S}\text{et}_{\Delta}$ can be endowed with the \textbf{Kan model structure} (in which case we denote $\mathcal{S}\text{et}_{\Delta}$ by $\Sigma$) for which cofibrations are monomorphisms, fibrations $f:X \rightarrow Y$ are \textbf{Kan fibrations}, that is for any diagram of solid arrows:
\beq
\setlength{\unitlength}{0.5cm}
\begin{picture}(6,5)(0,0)
\put(0,0){$\Delta^n$}
\put(4.5,0){$Y$}
\put(0,4){$\Lambda_i^n$}
\put(4.5,4){$X$}
\put(1.2,0.2){\vector(1,0){3}}
\put(0.2,3.8){\vector(0,-1){3}}
\put(1.2,4.3){\vector(1,0){3}}
\put(4.7,3.8){\vector(0,-1){3}}
\multiput(1,1)(0.2,0.2){15}{\circle*{0.1}}
\put(3.5,3.5){\vector(1,1){0.5}}
\put(5,2.5){$f$}
\end{picture}  \nonumber \\ \\
\eeq
one can find a dotted arrow as shown above that makes the diagram commutative for all $n$ and $i$, and finally a map $f:X \rightarrow Y$ is a weak equivalence if the induced map of geometric realizations $|X| \rightarrow |Y|$ is a homotopy equivalence of topological spaces.\\

The process whereby $\infty$-categories are produced from model categories is called $\textbf{localization}$. In the following definition categories are regarded as $\infty$-categories in an obvious way.
\begin{localization}(\cite{W})
Let $\mathcal{C}$ be an $\infty$-category, $S$ a set of morphisms of $\mathcal{C}$. A localization of $\mathcal{C}$ along $S$ is an $\infty$-category $L_S \mathcal{C}$ together with a functor of $\infty$-categories $l: \mathcal{C} \rightarrow L_S \mathcal{C}$ characterized by the following universal property: for any $\infty$-category $\mathcal{D}$, the map of internal Hom objects in the homotopy category of $\infty$-categories
\beq
\mathbb{R}\underline{\text{Hom}}(L_S\mathcal{C}, \mathcal{D}) \rightarrow \mathbb{R}\underline{\text{Hom}}(\mathcal{C}, \mathcal{D}) \nonumber
\eeq
induced by $l$ is fully faithful and its essential image is given by functors $F: \mathcal{C} \rightarrow \mathcal{D}$ that send morphisms in $S$ to an equivalence in $\mathcal{D}$.
\end{localization}
For $\mathcal{M}$ a model category, one lets $L \mathcal{M}$ be the localization of $\mathcal{M}$ along the set of weak equivalences of $\mathcal{M}$.\\

One important example of localization is provided as we will see below by the localization of combinatorial simplicial model categories. Those are model categories that are both combinatorial model categories and simplicial model categories, two concepts we introduce presently. To define simplicial model categories, we need the notions of \textbf{monoidal model category} and of \textbf{left Quillen bifunctor}.
\begin{Quillen} (\cite{L1})
If \textbf{A}, \textbf{B} and \textbf{C} are model categories, a functor $F: \textbf{A}\times \textbf{B} \rightarrow \textbf{C}$ is called a left Quillen bifunctor if $F$ preserves small colimits separately in each variable, and if for $i:A \rightarrow A'$ and $j:B \rightarrow B'$ cofibrations in \textbf{A} and \textbf{B} respectively, the induced map:
\beq
i \wedge j: F(A',B) \coprod_{F(A,B)}F(A, B') \rightarrow F(A', B') \nonumber
\eeq
is a cofibration in \textbf{C}, and further if either of $i$ or $j$ is a trivial cofibration (both a cofibration and a weak equivalence), then so is $i \wedge j$.\\
\end{Quillen}

\begin{MonModCat}(\cite{L1})
A monoidal model category is a monoidal category \textbf{M} equipped with a model structure for which the tensor product functor $\otimes: \textbf{M} \times \textbf{M} \rightarrow \textbf{M}$ is a left Quillen bifunctor, the monoidal structure on \textbf{M} is closed, and the unit object $1 \in \textbf{M}$ is cofibrant. Recall that any model category has an initial object $\emptyset$ and an object $X$ of such a category has a unique map $\emptyset \rightarrow X$ which makes $X$ cofibrant if this map is a cofibration.
\end{MonModCat}

One can enrich model categories over monoidal model categories. If \textbf{M} is a monoidal model category, a \textbf{M}-enriched model category is a model category \textbf{C} equipped with a model structure for which \textbf{C} is tensored and cotensored over itself, and the tensor product $\otimes: \textbf{C} \times \textbf{M} \rightarrow \textbf{C}$ is a left Quillen bifunctor.\\

%To be specific, a monoidal category $(\mathcal{C}, \otimes)$ is \textbf{right closed} if every functor $B \rightarrow B \otimes A$ admits a right adjoint $X \rightarrow X^A$. Dually, $\mathcal{C}$ is said to be \textbf{left closed} if for all $A$ in $\mathcal{C}$, the functor $B \rightarrow A \otimes B$ admits a right adjoint $X \rightarrow \:^AX$. Let $\mathcal{C}$ be a right-closed monoidal category and $\mathcal{D}$ a category enriched over $\mathcal{C}$. Suppose that for all objects $C \in \mathcal{C}$ and $X \in \mathcal{D}$ the functor
%\begin{align}
%\mathcal{D} &\rightarrow \mathcal{C}\\
%Y &\mapsto \text{Map}_{\mathcal{D}}(X,Y)^C  \nonumber
%\end{align}
%is representable, that is there exists an object $Z=:X \otimes C$ and an isomorphism of functors $\text{Map}_{\mathcal{D}}(X, - )^C \sim \text{Map}_{\mathcal{D}}(Z, -)$. Then we say $\mathcal{D}$ is \textbf{tensored over} $\mathcal{C}$. Dually, if for all $C \in \mathcal{C}$ and $X \in \mathcal{D}$ there is an object $\:^CX$ representing the functor $Y \mapsto \:^C \text{Map}_{\mathcal{D}}(Y,X)$, then we say that $\mathcal{C}$ is \textbf{cotensored over} $\mathcal{D}$.\\

\begin{SimpModCat} (\cite{L1})
If one endows the category $\mathcal{S}\text{et}_{\Delta}$ of simplicial sets with the cartesian product and the Kan model structure, it becomes a monoidal model category. A $\mathcal{S}\text{et}_{\Delta}$-enriched category is called a \textbf{simplicial model category}.\\
\end{SimpModCat}

\begin{WeaklySat}(\cite{L1})
Let $C$ be a category with all small colimits, $S$ a class of morphisms in $C$. One says that $S$ is \textbf{weakly saturated} if the following conditions are met: given a diagram
\beq
\begin{CD}
x @>>f> y\\
@VVV @VVV\\
x'@>>f'> y'
\end{CD} \nonumber
\eeq
such that $f$ belongs to $S$, then so does $f'$. If $x$ is an object of $C$, $\alpha$ is an ordinal, $\{z_{\beta}\}_{\beta<\alpha}$ is a system of objects of the undercategory $C_{x/}$, if for $\beta \leq \alpha$, $x_{\rightarrow \beta}$ denotes a colimit in $C_{x/}$ of the system $\{x_{\gamma}\}_{\gamma < \beta}$, if for each $\beta < \alpha$ the natural maps $x_{\rightarrow \beta} \rightarrow x_{\beta}$ belong to $S$, then so does the induced map $x \rightarrow x_{\rightarrow \alpha}$. Finally for a commutative diagram
\beq
\begin{CD}
x @>>> x' @>>> x\\
@VVfV @VVgV @VVfV\\
y @>>> y' @>>> y
\end{CD}  \nonumber
\eeq
for which both horizontal compositions are the identity, if $g$ is in $S$, then so is $f$.
\end{WeaklySat}

\begin{CombModCat} (\cite{L1})
A model category \textbf{C} is said to be a \textbf{combinatorial model category} if it is presentable, there exists a set $I$ of generating cofibrations such that the collection of all cofibrations in \textbf{C} is the smallest weakly saturated class of morphisms containing $I$, and if there exists a set $J$ of generating trivial cofibrations such that the collection of all trivial cofibrations in \textbf{C} is the smallest weakly saturated class of morphisms containing $J$.
\end{CombModCat}

One defines a \textbf{combinatorial simplicial model category} to be a model category that is both a combinatorial model category and a simplicial model category.

\begin{presentable} (\cite{W})
An $\infty$-category is said to be \textbf{presentable} if it is equivalent to the localization of a combinatorial simplicial model category.
\end{presentable}

\begin{Rmkpres}
One could also have defined a presentable $\infty$-category as done in \cite{L1}: an $\infty$-category is presentable if it is accessible and admits small colimits. Then for $\mathcal{C}$ an $\infty$-category, $\mathcal{C}$ is presentable if and only if it is equivalent to $N(\mathcal{A}^0)$, $\mathcal{A}^0$ the subcategory of fibrant-cofibrant objects in a combinatorial simplicial model category $\mathcal{A}$ (\cite{L1}).
\end{Rmkpres}

One denotes by $\text{Cat}_{\infty}^p$ the full subcategory of $\text{Cat}_{\infty}$ spanned by presentable $\infty$-categories.\\

An example of $\infty$-category that will be fundamental to us is provided by \textbf{symmetric monoidal $\infty$-categories}. We will follow \cite{W} closely to define these. We denote by $\Xi$ the category of pointed finite ordinals and point preserving maps. We denote a pointed ordinal $n \coprod \{\ast\}$ by $n\ast$. One has a monoidal category structure on $\Xi$ given by $(\Xi, \vee, 0\ast)$.
We denote by P($\Sigma$) the category of $\Sigma$-enriched precategories with the injective or Reedy model structure as considered in \cite{W} (see \cite{L1} for details on these model structures). Let $p:C \rightarrow \Xi$ be an object of P($\Sigma)_{/\Xi}$. An arrow $F$ in $C(a,b)$ is said to be \textbf{$p$-cocartesian} if for all $c \in C$, the induced morphism
\beq
C(b,c)\rightarrow C(a,c) \times_{\Xi(p(a),p(c))} \Xi(p(b),p(c)) \nonumber
\eeq
is a weak equivalence in $\Sigma$.\\

An object $p:C \rightarrow \Xi$ in P($\Sigma)/\Xi$ is said to be a \textbf{cofibered $\infty$-category} if for any morphism $f:n\ast \rightarrow m\ast$ in $\Xi$, for every object $c \in C$ such that $p(c)=n$ there exists a cocartesian morphism $F$ such that $p(F)$ is isomorphic to $f$ in the category $\Xi_{n\ast /}$.
For $p:C \rightarrow \Xi$ an object of P($\Sigma)_{/\Xi}$, we denote by $C_{n\ast}$ the fiber of $p$ at $n\ast \in \Xi$. For all $n \geq 1$ and $0<i \leq n$ one considers the $n$ pointed maps $p_i:n\ast \rightarrow 1\ast$ in $\Xi$ defined by $p_i(j)=\{j\}$ if $j=i$, $\ast$ otherwise. Those maps induce natural maps $(p_i)_{\uparrow}:C_{n\ast} \rightarrow C_{1\ast}$.
\begin{SymmMonCat}(\cite{W})
Let $C$ be an $\infty$-precategory. A symmetric monoidal $\infty$-category is a cofibered object $p:C \rightarrow \Xi$ of P($\Sigma)_{/\Xi}$ for which
\beq
C_{n\ast} \xrightarrow{\amalg_i (p_i)_{\uparrow}} (C_{1\ast})^n   \nonumber
\eeq
is an equivalence for all $n \geq 0$.
\end{SymmMonCat}

An important example of symmetric monoidal $\infty$-category is given in \cite{W} and is the following. The category P($\Sigma)$ of $\Sigma$-precategories with the injective or Reedy model structure is a model category for $\infty$-categories, and turns out to be a symmetric monoidal simplicial model category for the cartesian product, from which it follows that its localization with respect to weak equivalences, which is none other than the $\infty$-category Cat$_{\infty}$ of $\infty$-categories, is itself a symmetric monoidal $\infty$-category.\\

An arrow $f:n\ast \rightarrow m\ast$ in $\Xi$ is said to be \textbf{inert} if $f^{-1}\{i\}$ is a single element, $i \in m\ast -\ast$. An arrow $F$ in a symmetric monoidal $\infty$-category $C$ is said to be \textbf{p-inert} if $F$ is a cocartesian arrow in $C$ and $p(F)$ is inert in $\Xi$. For two symmetric monoidal categories $p:C \rightarrow \Xi$ and $q:D \rightarrow \Xi$, a functor $F:C \rightarrow D$ for which the diagram
\beq
\setlength{\unitlength}{0.5cm}
\begin{picture}(7,5)(0,0)
\put(1.5,1.5){$p$}
\put(5.5,1.5){$q$}
\put(3.5,4){$F$}
\put(0,3.3){$C$}
\put(6.2,3.3){$D$}
\put(3.2,0){$\Xi$}
\put(1,3){\vector(1,-1){2}}
\put(6,3){\vector(-1,-1){2}}
\put(1,3.5){\vector(1,0){4.8}}
\end{picture}  \nonumber \\ \\
\eeq
is commutative and for which $F$ carries $p$-cocartesian arrows to $q$-cocartesian arrows is said to be symmetric monoidal. It is said to be \textbf{lax symmetric monoidal} if it carries $p$-inert arrows to $q$-cocartesian arrows. One denotes by Cat$^{\otimes}_{\infty}$ the $\infty$-category of symmetric monoidal $\infty$-categories and symmetric monoidal functors.\\

For the purpose of studying Tannaka duality, we will restrict our attention to $\infty$-categories that are rigid since for those a notion of duality can be defined. One says that an object of a symmetric monoidal $\infty$-category $C$ is dualizable if it admits a dual in the homotopy category h$C$. A symmetric monoidal $\infty$-category for which all objects are dualizable is said to be \textbf{rigid}.

\section{$E_{\infty}$-rings}
Let $\mathcal{C}$ be an $\infty$-category with finite limits. One regards $\mathbb{Z}$ as a filtered category. Let $T$ be an endofunctor on $\mathcal{C}$. Consider:
\begin{align}
\phi:\mathbb{R}\underline{\text{Hom}}(\mathbb{Z},\mathcal{C}) &\rightarrow \mathbb{R}\underline{\text{Hom}}(\mathbb{Z},\mathcal{C}) \nonumber \\
F &\longmapsto \phi(F):n \mapsto \phi(F)(n)=T(F(n+1))  \nonumber
\end{align}
\begin{Tspectrum} (\cite{W})
Let $\mathcal{C}$ be an $\infty$-category with finite limits, $T$ an endofunctor on $\mathcal{C}$. A functor $F:\mathbb{Z} \rightarrow \mathcal{C}$ for which $F \rightarrow \phi(F)$ is an equivalence in $\mathbb{R}\underline{\text{Hom}}(\mathbb{Z},\mathcal{C})$ is called a $\mathbf{T}$-\textbf{spectrum} object of $\mathcal{C}$.
\end{Tspectrum}
Let $\mathcal{C}$ be a pointed $\infty$-category with finite limits. The \textbf{loop functor} $\Omega$ of $\mathcal{C}$ is the endofunctor of $\mathcal{C}$ defined by:
\beq
\Omega:x \mapsto 0 \times_x 0  \nonumber
\eeq
\begin{spectrum} (\cite{W})
If $\mathcal{S}_*$ denotes the $\infty$-category of pointed spaces, a \textbf{spectrum} is an object of $\text{Sp}:=\text{Sp}_{\Omega}(\mathcal{S}_*)$, the $\infty$-category of $\Omega$-spectrum objects of $\mathcal{S}_*$.
\end{spectrum}

For $\mathcal{C}$ a presentable $\infty$-category, the natural functor $\text{Ev}_n: \text{Sp}_T(\mathcal{C}) \rightarrow \mathcal{C}$ has a left adjoint $\text{Fr}_n : \mathcal{C} \rightarrow \text{Sp}_T (\mathcal{C})$. If $\ast$ denotes the final object of $\mathcal{S}$, $T=\Omega$ and $\mathcal{C}=\mathcal{S}$, then $\text{Fr}_0(\ast)=:\mathbb{S}$ is called the sphere spectrum. Sp can be endowed with a symmetric monoidal structure uniquely characterized by the conditions that $\mathbb{S}$ be the unit object of Sp and $\otimes: \text{Sp} \times \text{Sp} \rightarrow \text{Sp}$ preserves colimits separately in each variable (\cite{L2}).\\

One last ingredient needed before defining an $E_{\infty}$ ring is the concept of \textbf{commutative monoid object}.
\begin{ComMonObj}(\cite{W})
Let $p:C \rightarrow \Xi$ be a symmetric monoidal $\infty$-category. A commutative monoid object in $C$ is a lax symmetric monoidal section of $p$. One denotes by CMon($C$) the $\infty$-category of commutative monoid objects in $C$.
\end{ComMonObj}

\begin{CRingSpec}(\cite{W})
A commutative ring spectrum, or $E_{\infty}$-ring, is a commutative monoid object in the $\infty$-category of spectra endowed with the smash product monoidal structure.
\end{CRingSpec}
Now that we have seen what an $E_{\infty}$-ring is, we go back to the general theory of commutative monoid objects. In \cite{W} and \cite{L2}, for $R$ a commutative monoid object in a symmetric monoidal $\infty$-category $C$, a symmetric monoidal $\infty$-category $p:\text{Mod}_{R}(C) \rightarrow \Xi$ is defined, the $\infty$-category of $R$-modules of $C$. The reader is referred to the references for details pertaining to the construction of such a category. For our purposes we will retain that for such a symmetric monoidal $\infty$-category, the unit object is canonically equivalent to $R$ and one has a relative tensor product over $R$ denoted by $\otimes_R$, leading to a functor (\cite{W}):
\begin{align}
\text{Mod}(C): \text{CMon}(C) & \rightarrow \text{Cat}^{\otimes}_{\infty}  \nonumber \\
R & \mapsto \text{Mod}_R(C) \nonumber \\
f:R &\rightarrow S \mapsto -\otimes_R S  \nonumber
\end{align}

A symmetric monoidal $\infty$-category $C$ is said to be a \textbf{presentable symmetric monoidal $\infty$-category} if its symmetric product preserves small colimits separately in each variable and the fibers $C_{n\ast}$ are presentable $\infty$-categories for all $n > 0$.
\begin{Rlinear}(\cite{W})
Let $D$ be a presentable symmetric monoidal $\infty$-category. A presentable $\infty$-category is said to be \text{$R$-linear} if one can give it the structure of a Mod$_R(D)$-module object in Cat$^p_{\infty}$.
\end{Rlinear}

\section{$R$-Tannakian group stacks}
We will be interested in group stacks. In this regard we will first go over stacks as done in \cite{W} and then study Tannakian Hopf algebras. If $(C, \tau)$ is a site, a topology $\tau$ on the $\infty$-category $C$ is equivalent to having covering families for all objects of $C$, with the added axioms that those families  contain equivalences, and are stable under compositions and base change. Let $\{u_i \rightarrow x\}_{i \in I}$ be the covering family of some object $x \in C$, write $u=\prod_i u_i$ and define the simplicial object:
\beq
u_{\ast}:[n] \mapsto  u \times _x \cdots \times_x u   \nonumber
\eeq
\begin{stack} (\cite{W})
If $(C, \tau)$ is a site, $X$ an $\infty$-category with limits, an $X$-valued prestack $F:C^{op} \rightarrow X$ is an $X$-valued stack if for all $x \in C$ and all coverings $u_{\ast}$ the map $F(x) \rightarrow \lim_{\Delta}F(u_{\ast})$ is an equivalence in $X$.
\end{stack}

An $\mathcal{S}$-valued stack will simply be referred to as a \textbf{stack}, and the $\infty$-category of such stacks will be denoted by St$^{\tau}(C)$. Those topologies $\tau$ on $C$ for which every representable functor on $C$ is a stack with respect to $\tau$ are called \textbf{subcanonical} and will play an important role later. If one writes Aff$_C$ for CMon$(C)^{op}$ and one fixes a symmetric monoidal $\infty$-category $D$, then Aff$_R:=\text{Aff}_{\text{Mod}_R(D)}$ effectively corresponds to the opposite of the $\infty$-category of commutative $R$-algebras in $D$, where $R$-algebras in $D$ are simply defined to be commutative monoid objects in Mod$_R(D)$. This will be the category on which we will construct stacks. To be more precise, we would like to work with \textbf{group stacks}. That notion necessitates the introduction of the notion of group objects in $\mathcal{S}$. For an $\infty$-category $C$ with pullbacks, a group object $G$ is a functor $G: \Delta^{op} \rightarrow C$ such that for all $n \geq 0$ the canonical map $G([n]) \rightarrow G([1])\times _{G([0])} \cdots \times _{G([0])} G([1])$ is an equivalence in $C$, if it takes every partition $[2]=\{P \cup P' | P \cap P'=\{x\}, x \in P\}$ to a diagram
\beq
\begin{CD}
G([2]) @>>> G(P')\\
@VVV @VVV\\
G(P) @>>> G(\{x\})
\end{CD}  \nonumber
\eeq
and finally if $G([0])$ is a terminal object in $C$. Let Gp($\mathcal{S})$ be the $\infty$-category of group objects of $\mathcal{S}$. A Gp($\mathcal{S})$-valued stack will be called a group stack by virtue of the equivalence between Gp(Pr($C$)) and $\mathbb{R}\underline{\text{Hom}}(C^{\text{op}}, \text{Gp}(\mathcal{S}))$ (\cite{W}), and the $\infty$-category of group stacks on an $\infty$-category $C$ will be denoted by Gp$^{\tau}(C)$. One denotes by Gp$^{\tau}(R):=\text{Gp}^{\tau}(\text{Aff}_R)$ the $\infty$-category of group stacks on the site $(\text{Aff}_R, \tau)$ of commutative $R$-algebras.\\

Among the topologies one can put on Aff$_R$, in this paper we are interested in the \textbf{finite} topology (\cite{W}). Thus we have to define what are finite coverings. Let $R$ be an E$_{\infty}$-ring and $A \rightarrow B$ a map of $R$-algebras. Consider the base change functor
\begin{align}
B \otimes_A -:\text{Mod}_A & \rightarrow \text{Mod}_B   \nonumber\\
M & \mapsto B \otimes_A M  \nonumber
\end{align}
Such a map is said to be conservative if $B \otimes_A M \simeq 0$ if and only if $M \simeq 0$.\\

For $R$ an E$_{\infty}$-ring, $A$ an $R$-algebra, an $A$-module $M$ is said to be finite if the functor $- \otimes _A M: \text{Mod}_A \rightarrow \text{Mod}_A$ preserve all small limits.
A map $A \rightarrow B$ of $R$-algebras is said to be finite if $B$ is finite when considered as an $A$-module.
For $R$ an E$_{\infty}$-ring, a finite family of maps $\{A \rightarrow B_i\}_{i \in I}$ of $R$-algebras is said to be a finite covering if $A \rightarrow B_i$ is finite and conservative for all $i \in I$.
\begin{finite} (\cite{W})
Let $R$ be an E$_{\infty}$-ring. Finite coverings define a topology on Aff$_R$ referred to as the \textbf{finite topology}.
\end{finite}
\begin{Rmrk1}
It is implied in this notation that Aff$_R=\text{Aff}_{\text{Mod}_R(D)}$ with $D=$Sp.\\
\end{Rmrk1}

For $R$ an E$_{\infty}$ ring, the prestack
\begin{align}
\text{Mod}:\text{Aff}^{op}_R & \rightarrow \text{Cat}_{\infty}\label{stack}\\
A & \mapsto \text{Mod}_A   \nonumber\\
(A \rightarrow B) & \mapsto B \otimes_A -  \nonumber
\end{align}
is a stack of $\infty$-categories over Aff$_R$ with respect to the finite topology.

\begin{HopfRAlg}(\cite{W})
For $C$ a symmetric monoidal $\infty$-category, $R$ a commutative monoid object in $C$, a \textbf{Hopf $R$-algebra} in $C$ is a cogroup object in the symmetric monoidal $\infty$-category of commutative $R$-algebras in $C$.
\end{HopfRAlg}
\begin{finHopf}(\cite{W})
Let $C$ be a symmetric monoidal $\infty$-category. Let $(\text{Aff}_R, \tau)$ be a site. A Hopf $R$-algebra $B$ in $C$ is called a \textbf{$\tau$-Hopf $R$-algebra} if $B([1]) \in \tau(R)$
\end{finHopf}

Recall that we defined Aff$_C:=\text{CMon}(C)^{op}$ and if one uses the notation $C^{\wedge}$ for the $\infty$-category of prestacks $\mathbb{R}\underline{\text{Hom}}(C^{op}, \mathcal{S})$, then one denotes the Yoneda embedding Aff$_C \rightarrow (\text{Aff}_C)^{\wedge}$ by Spec. For $(C, \tau)$ a site, Pr($C):=\mathbb{R}\underline{\text{Hom}}(C^{op}, \mathcal{S})$, one can define a classifying prestack functor
\begin{align}
\overline{B}:\text{Gp}(\text{Pr}(C)) &\rightarrow \text{Pr}(C)  \nonumber\\
G &\mapsto \overline{B}G: x \mapsto \text{B}(G(x))   \nonumber
\end{align}
the stackification of which is called the classifying stack functor (\cite{W}):
\beq
\tilde{B}:\text{Gp}^{\tau}(C) \rightarrow \text{St}^{\tau}(C)  \nonumber
\eeq
For $C$ and $D$ symmetric monoidal $\infty$-categories, one writes End$^{\otimes}(F)=\text{Map}(F,F)$ in $\mathbb{R}\underline{\text{Hom}}^{\otimes}_{\Xi}(C,D)$.

\section{Connecting the dots}

\begin{RTensor}(\cite{W})
Let $R$ be an E$_{\infty}$-ring. A symmetric monoidal $\infty$-category that is $R$-linear, stable and presentable is referred to as an \textbf{$R$-tensor $\infty$-category}. One denotes by Tens$^{\otimes}_R$ the $\infty$-category of all such categories.
\end{RTensor}

The stack Mod considered in \eqref{stack} can be used to define a map
\begin{align}
\text{Mod}:\text{St}^{\tau}(R)^{op} & \rightarrow \text{Cat}_{\infty}  \nonumber\\
F & \mapsto \text{Mor}(F, \text{Mod})  \nonumber
\end{align}
where as done in \cite{W} we use the same notation Mod for this map as the stack Mod for simplicity of notation. The $\infty$-category Mor($F, \text{Mod})$ is equipped with a structure of $R$-tensor $\infty$-category and thus one has a well-defined functor $\text{Mod}: \text{St}^{\tau}(R)^{op} \rightarrow \text{Tens}^{\otimes}_R$. One denotes by Tens$^{\text{rig}}_R$ the $\infty$-category of \textbf{rigid $R$-tensor $\infty$-categories}. Following \cite{W}, if we restrict our attention to rigid objects, one has a stack Perf: Calg$_R \rightarrow \text{Tens}^{\text{rig}}_R$ that sends a commutative $R$-algebra $A$ to the $\infty$-category Mod$^{\text{rig}}_A$ of rigid $A$-modules, and the restriction of the map Mod leads to a functor:
\begin{align}
\text{Perf}: \text{St}^{\tau}(R)^{op} & \rightarrow \text{Tens}^{\text{rig}}_R  \nonumber\\
F &\mapsto \text{Mor}(F, \text{Perf}) \nonumber
\end{align}
This functor admits a left adjoint
\begin{align}
\text{Fib}: \text{Tens}^{\text{rig}}_R & \rightarrow \text{St}^{\tau}(R)^{op}  \nonumber\\
C &\mapsto \underline{\text{Hom}}(C, \text{Perf})  \nonumber
\end{align}
with Fib$(C)(A)=\text{Map}_{\text{Tens}^{\text{rig}}_R}(C, \text{Mod}^{\text{rig}}_A)$ for a commutative $R$-algebra $A$.\\

One defines $\hat{\omega}_G: \text{Mod}(\tilde{\text{B}}G) \rightarrow \text{Mod}_R$ and $\omega_G: \text{Perf}(\tilde{\text{B}}G) \rightarrow \text{Mod}^{\text{rig}}_R$.\\

Let $R$ be an $E_{\infty}$-ring. An affine group stack $G=\text{Spec}B$ in Gp$^{\tau}(R)$ is said to be weakly rigid if $\text{End}^{\otimes}(\hat{\omega}_G) \rightarrow \text{End}^{\otimes}(\omega_G)$ is an equivalence.\\

\begin{RTan}(\cite{W})
Let $R$ be an $E_{\infty}$-ring, $\tau$ a subcanonical topology. A group stack $G$ in Gp$^{\tau}(R)$ is said to be $R$-Tannakian if $G=\text{Spec}B$ for a Hopf $R$-algebra $B$, and is weakly rigid. It is further said to be finite-$R$-Tannakian if it is $R$-Tannakian for a finite-Hopf $R$-algebra.
\end{RTan}
Let TGp$^{\tau}(R)$ be the full subcategory of Gp$^{\tau}(R)$ spanned by the $R$-Tannakian group stacks.
As an aside, it is not the fact that we use a Hopf algebra that makes the group stack Tannakian, it is the weakly rigid condition. Rigidity is associated with a notion of duality which of course is what the Tannakian formalism is all about.\\

We now come to the part where both the $\infty$-category and the algebraic part get connected. For Tannaka duality to take place, we need to consider fiber functors. We are thus led to define (\cite{W}) \textbf{pointed rigid $R$-tensor $\infty$-categories}, pairs $(C, \omega)$ where $C$ is a rigid $R$-tensor $\infty$-category and $\omega:C \rightarrow \text{Mod}^{\text{rig}}_R$ is a $R$-tensor functor referred to as a fiber functor. One denotes by $(\text{Tens}^{\text{rig}}_R)_{\ast}$ the category spanned by such pairs. We also need the following map:
\begin{align}
\text{Fib}_{\ast}: (\text{Tens}^{\text{rig}}_R)_{\ast} & \rightarrow \text{Gp}^{\tau}(R)^{op} \nonumber\\
(C, \omega) &\mapsto \text{End}^{\otimes}(\omega)  \nonumber
\end{align}
That such a map is well-defined follows from a result of \cite{W} which states that End$^{\otimes}(\omega)$ is a representable Gp$(\mathcal{S})$-valued prestack, and thus an affine group stack with respect to $\textit{any}$ subcanonical topology $\tau$. Dually one has the functor
\begin{align}
\text{Perf}_{\ast}: \text{Gp}^{\tau}(R)^{op} & \rightarrow (\text{Tens}^{\text{rig}}_R)_{\ast}  \nonumber\\
G & \mapsto (\text{Perf}(\tilde{\text{B}}G), f^*)  \nonumber
\end{align}
where $f: \ast \rightarrow \tilde{\text{B}}G$ induces the functor $f^*: \text{Perf}(\tilde{\text{B}}G) \rightarrow \text{Mod}^{\text{rig}}_R$. Those functors give rise to an adjunction (\cite{W}):
\beq
\text{Fib}_{\ast}:(\text{Tens}^{\text{rig}}_R)_{\ast} \rightleftarrows \text{Gp}^{\tau}(R)^{op}:\text{Perf}_{\ast}  \nonumber
\eeq

Finally the pointed rigid $R$-tensor $\infty$-categories we work with have a Tannakian-like behavior for a choice of three topologies, one of which is of interest to us, the finite topology.
\begin{pointed}
Let $R$ be an $E_{\infty}$-ring. A \textbf{pointed $R$-Tannakian $\infty$-category} for the finite topology is a pair $(T, \omega)$ for $T$ a rigid $R$-tensor $\infty$-category and $\omega:T \rightarrow \text{Mod}^{\text{rig}}_R$ a finite fiber functor, meaning that the induced functor $\text{Mod}(\tilde{\text{B}}\text{End}^{\otimes}(\omega)) \rightarrow \text{Mod}_R$ is conservative and preserves small limits.
\end{pointed}
Once that is achieved, \cite{W} asserts that $(T, \omega)=\text{Perf}_{\ast}(G)$, thus displaying Tannakian group stacks as the generators of Tannakian $\infty$-categories.\\

\end{document}